\documentclass{amsart}
\usepackage{graphicx}
\usepackage{subfigure}
\newtheorem{theorem}{Theorem}[section]

\theoremstyle{definition}

\theoremstyle{remark}

\numberwithin{equation}{section}

\begin{document}

\title{High-order Algorithms for Riesz Type Turbulent Diffusion Equation}

%    Remove any unused author tags.

%    author one information
\author{Hengfei Ding}
\address{Department of Mathematics, Shanghai University, Shanghai 200444, China.}
%\curraddr{Tianshui Normal University}
\email{dinghf05@163.com}
\thanks{}

%    author two information
\author{Changpin Li}
\address{Department of Mathematics, Shanghai University, Shanghai 200444, China.}
\curraddr{ }
\email{lcp@shu.edu.cn}
\thanks{The work was partially supported by the National
Natural Science Foundation of China under Grant No. 11372170,
Key Program of Shanghai Municipal Education Commission
under Grant No. 12ZZ084, and the grant of ``The First-class
Discipline of Universities in Shanghai".}



\date{}

\dedicatory{}

\begin{abstract}
Numerical methods for fractional calculus attract increasing
interests due to its wide applications in various fields such as
physics, mechanics, etc. In this paper, we focus on constructing
high-order algorithms for Riesz derivatives, where the convergence
orders cover from the second order to the sixth order. Then we
apply the established schemes to the Riesz type 
turbulent diffusion equation (or, Riesz space fractional turbulent diffusion equation). Numerical experiments are displayed
which support the theoretical analysis.
\end{abstract}

\maketitle

\section{Introduction}

In recent years, fractional calculus has attracted
increasing interests due to its potential applications in physics,
mechanics, etc. For more details, see the review articles \cite{met}
and the recent studies \cite{ervin,tad,wang,mus,DengHesthaven2003,jin,zay}, and references cited therein. As far as we know,
the Riemann-Liouville (RL) derivative and Caputo derivative are commonly
used, which are respectively defined below

\begin{displaymath}
\,_{RL}D_{a,x}^\alpha f(x)
=\frac{1}{\Gamma(n-\alpha)}\frac{d^n}
{dx^n}\int_{a}^{x}(x-\xi)^{n-\alpha-1}f(\xi)d\xi,\;x\in(a,b),\;n-1<\alpha<n\in Z^{+},
\end{displaymath}
and,
\begin{displaymath}
\,_{C}D_{a,x}^\alpha f(x)
=\frac{1}{\Gamma(n-\alpha)}
\int_{a}^{x}(x-\xi)^{n-\alpha-1}f^{(n)}(\xi)d\xi,\;x\in(a,b),\;n-1<\alpha<n\in Z^{+}.
\end{displaymath}
These two definitions are generally not equivalent.

A linear combination of the left RL derivative (defined above) and the right RL derivative defined below
\begin{displaymath}
\,_{RL}D_{x,b}^\alpha f(x)
=\frac{(-1)^n}{\Gamma(n-\alpha)}\frac{d^n}
{dx^n}\int_{x}^{b}(\xi-x)^{n-\alpha-1}f(\xi)d\xi,\;x\in(a,b),\;n-1<\alpha<n\in Z^{+},\vspace{0.3cm}\\
\end{displaymath}
defines a new derivative, i.e., Riesz derivative,
$$
\begin{array}{ccc}\label{2}
\displaystyle \frac{d^\alpha
f(x)}{d{|x|^\alpha}}=-\frac{1}{2\cos\left(\pi\alpha/2\right)}\left
(\,_{RL}D_{a,x}^\alpha+\,_{RL}D_{x,b}^\alpha\right)f(x).
\end{array}
$$

The left RL derivative reflects the memory (or, the dependence on the history),
while the right RL derivative ``looks forward to'' the future (or, the dependence upon the future). So the
Riesz derivative value of $f(x)$ at $x$ relies on the whole space $(a,b)$, but the values are different at different $x\in(a,b)$.
 From an angle of application, the case with $\alpha\in(0,2)$ is mostly often attracted attention.
 If $\alpha=0$, then $ \frac{d^\alpha f(x)}{d{|x|^\alpha}}=f(x)$; if $\alpha=1$, generally set $\frac{d^\alpha f(x)}{d{|x|^\alpha}}=f'(x)$ for convenience; if $\alpha=2$, then it is reduced to $f''(x)$.

From the studies available, the high-order algorithms for RL derivatives were established in \cite{lub}. %The high-order algorithms for Caputo derivatives were firstly constructed in \cite{5}.
%In \cite{liding1}, Ding et al. applied Fourier analysis to deriving the fourth-order scheme for Riesz derivative and numerically studied the Riesz space fractional diffusion equation.
In \cite{liding2}, Ding et al.
 %continued to
 established the sixth-order, eighth-order, tenth-order and twelfth-order schemes for the Riesz derivative by using the Fourier analysis. Then they used the former two schemes to the Riesz space fractional reaction-dispersion equation where the rigorous error analyses were given. The other odd-order (third-order, fifth-order, seventh-order, ninth-order, eleventh-order) schemes can be established from \cite{liding2} by choosing the different parameters therein. In spite of these, it is absolutely necessary to construct intuitionist and straightforward schemes for the Riesz derivatives. Here we find an interesting and enlightening way to establish high-order schemes (from 2nd-order to 6th-order) by using the corresponding generating functions \cite{lub}. Then we use these schemes to solve the Riesz space fractional turbulent diffusion equation.

In the following, we briefly introduce fractional modelling in this respect.
To the best of our knowledge, from the known first Fick's law
$$
\begin{array}{rrr}
\displaystyle {J}(x,t)=d_1
u(x,t)-d_2\frac{\partial u(x,t)}{\partial x} ,\;d_1>0,\;d_2>0,
\end{array}
$$
one can obtain the following advection-diffusion equation,
$$
\begin{array}{rrr}
\displaystyle\frac{\partial u(x,t)}{\partial t}
=-d_1\frac{\partial
u(x,t)}{\partial x}+d_2\frac{\partial^2u(x,t)}{\partial x^2} .
\end{array}
$$

If the advection-diffusion process at any position $x\in(a,b)$ relies on the whole space $(a,b)$ (i.e., long-range interactions),
 then the classical Fick's law does not work well yet. However, the fractional derivative can well characterize
  such long-range interactions. Now we generalize the typical Fick's law to a fractional version,
$$
\begin{array}{rrr}
\displaystyle {J}_\alpha(x,t)=d_1 u(x,t)-d_2\frac{\partial u(x,t)}{\partial
x}-d_{\alpha}\frac{\partial^{\alpha-1}  u(x,t)}{\partial
|x|^{\alpha-1}},\;\;\alpha\in(0,1),
\end{array}
$$
where $d_1$, $d_2$, $d_{\alpha}$ are positive constants. It follows that the
Riesz space fractional turbulent diffusion equation (or, Riesz type turbulent diffusion equation) is obtained,
$$
\begin{array}{rrr}
\displaystyle \frac{\partial u(x,t)}{\partial t} =-d_1\frac{\partial
u(x,t) }{\partial x}+d_2\frac{\partial^2 u(x,t)}{\partial
x^2}+d_{\alpha}\frac{\partial^{\alpha} u(x,t)}{\partial
|x|^{\alpha}}.
\end{array}
$$

If there is a source term, then one has
\begin{equation}
\displaystyle \frac{\partial u(x,t)}{\partial t} =-d_1\frac{\partial
u(x,t) }{\partial x}+d_2\frac{\partial^2 u(x,t)}{\partial
x^2}+d_{\alpha}\frac{\partial^{\alpha} u(x,t)}{\partial
|x|^{\alpha}}+s(x,t),\label{eq.1}
\end{equation}
where the Riesz partial derivative with
order $\alpha\in(0,1)$ is given by
\begin{equation}
\displaystyle \frac{\partial^\alpha
u(x,t)}{\partial{|x|^\alpha}}=-\frac{1}{2\cos\left(\pi\alpha/2\right)}\left
(\,_{RL}D_{a,x}^\alpha+\,_{RL}D_{x,b}^\alpha\right)u(x,t),\;0<\alpha<1.\label{eq.2}
\end{equation}
%See also \cite{ros} for more details.
Equ. (\ref{eq.1}) is subject to the following initial value condition,
$$
\begin{array}{rrr}
\displaystyle u(x,0)=\varphi(x),\;x\in[a,b],
\end{array}
$$
and the boundary value conditions (here choosing the homogeneous condition for brevity),
$$
\begin{array}{rrr}
\displaystyle u(a,t)=0,\;\; u(b,t)=0
,\;\;\;\;\;t\geq 0.
\end{array}
$$

The remainder of this paper is outlined as follows. In Section 2,
 five kinds of high-order (2nd-order, 3rd-order, $\cdots$, 6th-order) algorithms for the Riesz derivatives
  are developed. In the next section, we apply the derived
   schemes to the Riesz-type turbulent diffusion equation (\ref{eq.1}). Here we only use the 2nd-order,
    4th-order, 6th-order schemes to (\ref{eq.1}). The convergence orders are $\mathcal{O}(\tau^2+h^2)$, $\mathcal{O}(\tau^2+h^4)$,
     and $\mathcal{O}(\tau^2+h^6)$, where $\tau$ and $h$ are temporal and spatial stepsizes, respectively.
   In Section 4, numerical examples are displayed which support the theoretical analysis. The last section concludes this article.

\section{High-order numerical schemes}

If
$f^{(k)}(a+)=0\;(k=0,1,\ldots,p-1)$, then it follows from \cite{lub} that the left RL derivative has the
following approximations
\begin{equation}
\displaystyle \,_{RL}{{D}}_{a,x}^{\alpha}f(x)=
\frac{1}{h^{\alpha}}\sum\limits_{\ell=0}^{\infty}\varpi_{p,\ell}^{(\alpha)}f(x-\ell
h)+\mathcal {O}(h^p),\label{eq.3}
\end{equation}
in which $h$ is the steplength. Here we only show interests in $p=2,3,4,5,6$.
% Much higher-order approximations
%can be found from \cite{liwu}.

 The convolution (or weight) coefficients $\varpi_{p,\ell}^{(\alpha)}$ in
the above equations are those of the Taylor series expansions of the
corresponding generating functions $W_p^{(\alpha)}(z)$,
$$
\begin{array}{lll}
\displaystyle W_p^{(\alpha)}(z)=
\sum\limits_{\ell=0}^{\infty}\varpi_{p,\ell}^{(\alpha)}z^\ell, \;\alpha\in(0,2),
\end{array}
$$
where
$$\begin{array}{lll}
\displaystyle
W_2^{(\alpha)}(z)=\left(\frac{3}{2}-2z+\frac{1}{2}z^2\right)^{\alpha},\vspace{0.2 cm}\\
 \displaystyle
W_3^{(\alpha)}(z)=\left(\frac{11}{6}-3z+\frac{3}{2}z^2-\frac{1}{3}z^3\right)^{\alpha},\vspace{0.2 cm}\\
 \displaystyle
W_4^{(\alpha)}(z)=\left(\frac{25}{12}-4z+3z^2-\frac{4}{3}z^3+\frac{1}{4}z^4\right)^{\alpha},\vspace{0.2 cm}\\
 \displaystyle
W_5^{(\alpha)}(z)=\left(\frac{137}{60}-5z+5z^2-\frac{10}{3}z^3+\frac{5}{4}z^4-
\frac{1}{5}z^5\right)^{\alpha},\vspace{0.2 cm}\\
 \displaystyle
W_6^{(\alpha)}(z)=\left(\frac{147}{60}-6z+\frac{15}{2}z^2-\frac{20}{3}z^3+\frac{15}{4}z^4-
\frac{6}{5}z^5+\frac{1}{6}z^6\right)^{\alpha}.
\end{array}
$$

By tedious but direct calculations, one has
$$
\begin{array}{lll}
 \displaystyle\varpi_{2,\ell}^{(\alpha)}=\displaystyle
\left(\frac{3}{2}\right)^{\alpha}\sum\limits_{\ell_1=0}^{\ell}
\left(\frac{1}{3}\right)^{\ell_1}\varpi_{1,\ell_1}^{(\alpha)}
\varpi_{1,\ell-\ell_1}^{(\alpha)},
\end{array}
$$
$$\begin{array}{lll}
\displaystyle
\varpi_{3,\ell}^{(\alpha)}=\left(\frac{11}{6}\right)^{\alpha}
\sum\limits_{\ell_1=0}^{\ell}
\sum\limits_{\ell_2=0}^{\left[\frac{1}{2}\ell_1\right]}
\left(-1\right)^{\ell_2}\left(\frac{7}{11}\right)^{\ell_1-\ell_2}
\left(\frac{2}{7}\right)^{\ell_2}
\frac{(\ell_1-\ell_2)!}{\ell_2!(\ell_1-2\ell_2)!}
\varpi_{1,\ell-\ell_1}^{(\alpha)}\varpi_{1,\ell_1-\ell_2}^{(\alpha)},
\end{array}
$$
$$\begin{array}{lll}
\displaystyle
\varpi_{4,\ell}^{(\alpha)}\hspace{-0.3cm}&=&\hspace{-0.3cm}\displaystyle\left(\frac{25}{12}\right)^{\alpha}
 \sum\limits_{\ell_1=0}^{\ell}
\sum\limits_{\ell_2=0}^{\left[\frac{2}{3}\ell_1\right]}
\sum\limits_{\ell_3=\max\{0,2\ell_2-\ell_1\}}^{\left[\frac{1}{2}\ell_2\right]}\left(-1\right)^{\ell_2}
\left(\frac{23}{25}\right)^{\ell_1-\ell_2}\left(\frac{13}{23}\right)^{\ell_2-\ell_3}\left(\frac{3}{13}\right)^{\ell_3}
\times\vspace{0.2 cm}\\
 \displaystyle
&& \displaystyle
\frac{\left(\ell_1-\ell_2\right)!}{\ell_3!\left(\ell_2-2\ell_3\right)!\left(\ell_1+\ell_3-2\ell_2\right)!}
\varpi_{1,\ell-\ell_1}^{(\alpha)}\varpi_{1,\ell_1-\ell_2}^{(\alpha)},
\end{array}
$$
$$\begin{array}{lll}
\varpi_{5,\ell}^{(\alpha)}
\hspace{-0.3cm}&=&\hspace{-0.3cm}\displaystyle\left(\frac{137}{60}\right)^{\alpha}
 \sum\limits_{\ell_1=0}^{\ell}
\sum\limits_{\ell_2=0}^{\left[\frac{3}{4}\ell_1\right]}
\sum\limits_{\ell_3=\max\{0,2\ell_2-\ell_1\}}^{\left[\frac{2}{3}\ell_2\right]}
\sum\limits_{\ell_4=\max\{0,2\ell_3-\ell_2\}}^{\left[\frac{1}{2}\ell_3\right]}
\left(-1\right)^{\ell_2}
\left(\frac{163}{137}\right)^{\ell_1-\ell_2}
\left(\frac{137}{163}\right)^{\ell_2-\ell_3}
\vspace{0.2 cm}\\
 &&\hspace{-0.3cm}\displaystyle\times
\left(\frac{63}{137}\right)^{\ell_3-\ell_4}
\left(\frac{4}{21}\right)^{\ell_4}
\frac{\left(\ell_1-\ell_2\right)!\;}{\ell_4!\left(\ell_3-2\ell_4\right)
!\left(\ell_1+\ell_3-2\ell_2\right)!\left(\ell_2+\ell_4-2\ell_3\right)!}\varpi_{1,\ell-\ell_1}^{(\alpha)}
\varpi_{1,\ell_1-\ell_2}^{(\alpha)},
\end{array}
$$
and
$$\begin{array}{lll}
\displaystyle
\varpi_{6,\ell}^{(\alpha)}\hspace{-0.3cm}&=&\hspace{-0.3cm}\displaystyle\left(\frac{147}{60}\right)^{\alpha}
 \sum\limits_{\ell_1=0}^{\ell}
\sum\limits_{\ell_2=0}^{\left[\frac{4}{5}\ell_1\right]}
\sum\limits_{\ell_3=\max\{0,2\ell_2-\ell_1\}}^{\left[\frac{3}{4}\ell_2\right]}
\sum\limits_{\ell_4=\max\{0,2\ell_3-\ell_2\}}^{\left[\frac{2}{3}\ell_3\right]}
\sum\limits_{\ell_5=\max\{0,2\ell_4-\ell_3\}}^{\left[\frac{1}{2}\ell_4\right]}
\vspace{0.2 cm}\\
 &&\displaystyle
 \left.\left(-1\right)^{\ell_2}
 \left(\frac{213}{147}\right)^{\ell_1-\ell_2}
\left(\frac{237}{213}\right)^{\ell_2-\ell_3}
\left(\frac{163}{237}\right)^{\ell_3-\ell_4}
\left(\frac{62}{163}\right)^{\ell_4-\ell_5}
\left(\frac{5}{31}\right)^{\ell_5}\times \right.
\vspace{0.2 cm}\\
 &&\displaystyle
\frac{\left(\ell_1-\ell_2\right)!\;}{\ell_5!\left(\ell_4-2\ell_5\right)
!\left(\ell_1+\ell_3-2\ell_2\right)!\left(\ell_2+\ell_4-2\ell_3\right)!
\left(\ell_3+\ell_5-2\ell_4\right)! }\varpi_{1,\ell-\ell_1}^{(\alpha)}\varpi_{1,\ell_1-\ell_2}^{(\alpha)}
,\vspace{0.5 cm}\\
&&\ell=0,1,\ldots.
\end{array}
$$
Here $\varpi_{1,j}^{(\alpha)}$ is the first order coefficients defined by
$\varpi_{1,j}^{(\alpha)}=(-1)^j\frac{\Gamma(1+\alpha)}{\Gamma(j+1)\Gamma(1+\alpha-j)}$, $j=0,1,\cdots$.
If $j\ge 2$, then $\varpi_{1,j-1}^{(\alpha)}\le\varpi_{1,j}^{(\alpha)}$ for $\alpha\in (0,1)$
whilst $\varpi_{1,j-1}^{(\alpha)}\ge\varpi_{1,j}^{(\alpha)}$ for $\alpha\in (1,2)$. See \cite{liding3} for more information.

On the other hand, if
$f^{(k)}(b-)=0\;(k=0,1,\ldots,p-1)$, then one has the approximations below,
\begin{equation}
\displaystyle \,_{RL}{{D}}_{x,b}^{\alpha}f(x)=
\frac{1}{h^{\alpha}}\sum\limits_{\ell=0}^{\infty}\varpi_{p,\ell}^{(\alpha)}f(x+\ell
h)+\mathcal {O}(h^p),\;\; p=2,\cdots,6,\label{eq.4}
\end{equation}
where $h$ is also the stepsize.

Based on (\ref{eq.3}) and (\ref{eq.4}), if $f(x)$, together with its derivatives, has homogeneous boundary value conditions, one
easily gets
\begin{equation}
\displaystyle \frac{\partial^\alpha f(x)}{\partial{|x|^\alpha}}=
-\frac{1}{2\cos\left(\pi\alpha/2\right)h^\alpha}\sum\limits_{\ell=0}^{\infty}\varpi_{p,\ell}^{(\alpha)}\left
(f(x-\ell
h)+f(x+\ell
h)\right)+\mathcal {O}(h^p).\label{eq.5}
\end{equation}
Since $\alpha\in(0,1)$ is commonly used, we limit our interests in $\alpha\in (0,1)$.
  The case $\alpha\in(1,2)$ can be similarly studied.
When $\alpha=1$, we often set $\frac{\partial^\alpha f(x)}{\partial{|x|^\alpha}}=f'(x)$
 which is the trival case so is omitted here.

 The properties of the convolution coefficients $\varpi_{p,\ell}^{(\alpha)}$ are very important for constructing
 effective numerical algorithms for Riemann-Liouville time fractional differential equations and
 Riemann-Liouville (or Riesz) space fractional differential equations. So we first study these convolution
 coefficients.
 For the space fractional differential
equations, we use the following properties of the
coefficients $\varpi_{p,\ell}^{(\alpha)}$ to show the stability and convergence of the
derived algorithms.
Here we focus on studying the case $\alpha\in(0,1)$.

\begin{theorem}\label{th.2} For $0<\alpha<1$, then the following
inequalities hold:
$$
\begin{array}{lll}
\displaystyle
\sum\limits_{\ell=0}^{\infty}\varpi_{p,\ell}^{(\alpha)}\cos(\ell\theta)\geq0,\;\theta\in[-\pi,\pi],\;p=2,3,5,6.
\end{array}
$$
\end{theorem}

\begin{proof}  We only prove $p=2$, the other cases can be almost similarly shown so are left out here. Let
$$f_1(\alpha,\theta)=\sum\limits_{\ell=0}^{\infty}\varpi_{2,\ell}^{(\alpha)}\cos(\ell\theta),$$
which can be expanded as

$$
\begin{array}{lll}
\displaystyle
f_1(\alpha,\theta)=\sum\limits_{\ell=0}^{\infty}\varpi_{2,\ell}^{(\alpha)}\cos(\ell\theta)
=\frac{1}{2}\sum\limits_{\ell=0}^{\infty}\varpi_{2,\ell}^{(\alpha)}\left(\exp(i\ell\theta)+\exp(-i\ell\theta)
\right)\\
%\displaystyle=\frac{1}{2}\left[
%\left(\frac{3}{2}-2\exp(i\theta)+\frac{1}{2}\exp(2i\theta)\right)^{\alpha}
%+\left(\frac{3}{2}-2\exp(-i\theta)+\frac{1}{2}\exp(-2i\theta)\right)^{\alpha}\right] \\
\displaystyle =\frac{1}{2}\left[\left(1-\exp(i\theta)\right)^\alpha
\left(\frac{3}{2}-\frac{1}{2}\exp(i\theta)\right)^{\alpha}
+\left(1-\exp(-i\theta)\right)^\alpha
\left(\frac{3}{2}-\frac{1}{2}\exp(-i\theta)\right)^{\alpha}\right].
\end{array}
$$
Note that $f_1(\alpha,\theta)$ is a real-value and even function, so
we need only consider $\theta\in[0,\pi]$.

 Using the following equations
$$
\begin{array}{lll}
\displaystyle \left(1-\exp(\pm
i\theta)\right)^\alpha=\left(2\sin\frac{\theta}{2}\right)^\alpha\exp\left(\pm
i\alpha\left(\frac{\theta-\pi}{2}\right)\right)
\end{array}
$$
and
$$
\begin{array}{lll}
\displaystyle \left(x-yi\right)^\alpha=\left(x^2+y^2\right)
^{\frac{\alpha}{2}}\exp\left( i\alpha \phi\right),\;\;\phi=-\arctan\frac{y}{x},
\end{array}
$$
we can rewrite $f_1(\alpha,\theta)$ as
$$
\begin{array}{lll}
\displaystyle
f_1(\alpha,\theta)=\left(2\sin\frac{\theta}{2}\right)^\alpha\left(\lambda_1^2(\theta)+\mu_1^2(\theta)\right)
^{\frac{\alpha}{2}}\cos\alpha\left(\frac{\theta-\pi}{2}+\phi_1\right),
\end{array}
$$
where
$$
\begin{array}{lll}
\displaystyle
\lambda_1(\theta)=3-\cos\theta,\;\;\;\mu_1(\theta)=\sin\theta,\;\;\phi_1=-\arctan\frac{\mu_1(\theta)}{\lambda_1(\theta)}.
\end{array}
$$

Let
$$
\begin{array}{lll}
\displaystyle
z(\theta)=\frac{\theta-\pi}{2}+\phi_1,\;\;0\leq\theta\leq\pi.
\end{array}
$$
Then
$$
\begin{array}{lll}
\displaystyle z'(\theta)=\left(\frac{\theta-\pi}{2}+\phi_1\right)'
=\frac{3\sin^2\left(\frac{\theta}{2}\right)}{1+3\sin^2\left(\frac{\theta}{2}\right)}\geq0.
\end{array}
$$
Hence $z(\theta)$ is an increasing function in $[0,\pi]$ and
$$
\begin{array}{lll}
\displaystyle z_{\min}(\theta)=
z(0)=-\frac{\pi}{2},\,\;z_{\max}(\theta)=z(\pi)=0.
\end{array}
$$
Sot $\alpha\in(0,1)$ and $\theta\in[0,\pi]$ imply
$\cos\alpha\left(\frac{\theta-\pi}{2}+\phi_1\right)\geq0$. Furthermore one has
$$
\begin{array}{lll}
\displaystyle
f_1(\alpha,\theta)=\left(2\sin\frac{\theta}{2}\right)^\alpha\left(\lambda_1^2(\theta)+\mu_1^2(\theta)\right)
^{\frac{\alpha}{2}}\cos\alpha\left(\frac{\theta-\pi}{2}+\phi_1\right)\geq0.
\end{array}
$$
All this ends the proof.
\end{proof}

For $p=4$, $\alpha$ can not attain to $1$. But we have following theorem.

\begin{theorem}\label{th.3} If $\displaystyle 0<\alpha\leq\frac{\pi}
{\pi-\arccos\frac{1}{5}+2\arctan\frac{191\sqrt{6}}{317}}\approx 0.8439,$ then
the following inequality holds:
$$
\begin{array}{lll}
\displaystyle
\sum\limits_{\ell=0}^{\infty}\varpi_{4,\ell}^{(\alpha)}\cos(\ell\theta)\geq0,\;\theta\in[-\pi,\pi].
\end{array}
$$
\end{theorem}
\begin{proof} Let
$f_2(\alpha,\theta)=\sum\limits_{\ell=0}^{\infty}\varpi_{4,\ell}^{(\alpha)}\cos(\ell\theta).$
By almost the same reasoning as that of Theorem \ref{th.2}, we can get
$$
\begin{array}{lll}
\displaystyle
f_2(\alpha,\theta)=\left(2\sin\frac{\theta}{2}\right)^\alpha\left(\lambda^2_2(\theta)+\mu^2_2(\theta)\right)
^{\frac{\alpha}{2}}\cos\alpha\left(\frac{\theta-\pi}{2}+\phi_2\right),
\end{array}
$$
where
$$
\begin{array}{lll}
\displaystyle \;\;\;\lambda_2(\theta)
=25-23\cos\theta+13\cos2\theta-3\cos3\theta,\\ \displaystyle
\mu_2(\theta)=23\sin\theta-13\sin2\theta+3\sin3\theta,\;\;\phi_2=-\arctan\frac{\mu_2(\theta)}{\lambda_2(\theta)}.
\end{array}
$$
Since
$$\lambda_2(\theta)=14\left(\cos(\theta)-\frac{1}{2}\right)^2
+24\cos^2(\theta)\sin^2\left(\frac{\theta}{2}\right)+\frac{17}{2}>0,$$
and
$$\mu_2(\theta)=\left[12\left(\cos(\theta)-\frac{13}{12}\right)^2+\frac{71}{24}\right]\sin(\theta)\geq0,$$
so $\phi_2\in[-\frac{\pi}{2},0]$. We need only consider
$0\leq\theta\leq\pi $, therefore
$-\pi\leq\alpha\left(\frac{\theta-\pi}{2}+\phi_2\right)\leq0$.

Obviously, if
$\cos\alpha\left(\frac{\theta-\pi}{2}+\phi_2\right)\geq0$, then
$f_2(\alpha,\theta)\geq0$. A sufficient condition for
$\cos\alpha\left(\frac{\theta-\pi}{2}+\phi_2\right)\geq0$ is
$$
\begin{array}{lll}
\displaystyle
-\frac{\pi}{2}\leq\min_{\theta\in[0,\pi]}\alpha\left(\frac{\theta-\pi}{2}+\phi_2\right)\leq0,
\end{array}
$$
i.e.,
$$
\begin{array}{lll}
\displaystyle 0<\alpha\leq
\min_{\theta\in[0,\pi]}\left\{\frac{\pi}{\pi-\theta-2\phi_2}\right\}.
\end{array}
$$

Let
$
\displaystyle y(\theta)=\pi-\theta-2\phi_2$, then
$$
\begin{array}{lll}
\displaystyle y'(\theta)=\frac{1920
\left(5\cos\theta-1\right)\sin^4\left(\frac{\theta}{2}\right)}{a_2^2(\theta)+b_2^2(\theta)}.
\end{array}
$$
It is clear that $\theta=\arccos\frac{1}{5}$ is a unique maximum
point of $y(\theta)$ when $\theta\in[0,\pi]$, i.e.,
$$
\begin{array}{lll}
\displaystyle
y_{\max}\left(\theta\right)=y_{\max}\left(\arccos\frac{1}{5}\right)
=\pi-\arccos\frac{1}{5}+2\arctan\frac{191\sqrt{6}}{317},
\end{array}
$$
it follows that
$$
\begin{array}{lll}
\displaystyle
\min_{\theta\in[0,\pi]}\left\{\frac{\pi}{\pi-\theta-2\phi_2}\right\}=
\frac{\pi}
{\pi-\arccos\frac{1}{5}+2\arctan\frac{191\sqrt{6}}{317}}\approx0.8439,
\end{array}
$$
i.e.,
$$
\begin{array}{lll}
\displaystyle 0<\alpha\leq0.8439.
\end{array}
$$
 This finishes the proof.~\end{proof}
%The monotonicity of the second-order coefficients $\varpi^{(\alpha)}_{2,\ell}$ are
%often used to prove the stability and convergence of the constructed algorithms for
%the time fractional differential equations.
%
Theorems \ref{th.2} and \ref{th.3} are very suitable for numerically analyzing Riemann-Liouville
 space fractional partial differential equations. But for numerically analyzing Riemann-Liouville time
 fractional partial differential equations, the monotonicity of the coefficients $\varpi_{p,\ell}^{(\alpha)}$ $(p=1,2,3,4,5,6)$
 is often necessary. As far as we know, only the monotonicity of the coefficients $\varpi_{1,\ell}^{(\alpha)}$ for $\alpha\in(0,1)$
 and $\alpha\in(1,2)$, $\varpi_{2,\ell}^{(\alpha)}$ for $\alpha\in(0,1)$
 are available \cite{liding3}. In the following, we study the rest cases.

% The monotonicity of the coefficients $\varpi_{p,\ell}^{(\alpha)}$ $(p=2,3,4,5,6)$
%with respect to $\ell$ $(\ell=0,1,\ldots)$ are often used for stability and convergence analysis for the time fractional differential equations.
%The second-order coefficients $\varpi_{2,\ell}^{(\alpha)}$ have interesting monotonicity properties some of
%which (for $\alpha\in (0,1)$) have been studied in \cite{liding3}. Here we have the further results.
%%{\color{red}
{
\begin{theorem}\label{th.1} The second-order coefficients
$\varpi_{2,\ell}^{(\alpha)}\;(\ell=0,1,\ldots)$
satisfy
$$
\begin{array}{lll}
\displaystyle
(1)~\varpi_{2,0}^{(\alpha)}=\left(\frac{3}{2}\right)^{\alpha}>0,\;
\;
\varpi_{2,1}^{(\alpha)}=-\frac{4\alpha}{3}\left(\frac{3}{2}\right)^{\alpha}<0,\;\vspace{0.2 cm}\\
\displaystyle \hspace{0.6cm}
\varpi_{2,2}^{(\alpha)}=\frac{\alpha\left(8\alpha-3\right)}{9}\left(\frac{3}{2}\right)^{\alpha}
,\;\;\varpi_{2,3}^{(\alpha)}=-\frac{4\alpha\left(\alpha-1\right)\left(8\alpha-7\right)}
{81}\left(\frac{3}{2}\right)^{\alpha},\vspace{0.2 cm}\\
\displaystyle \hspace{0.6cm}
\varpi_{2,4}^{(\alpha)}=\frac{\alpha\left(\alpha-1\right)\left(64\alpha^2-176\alpha+123\right)}
{486}\left(\frac{3}{2}\right)^{\alpha}, \\ \hspace{3cm}\vdots
\\\hspace{3cm} \\
 (2)~{\textit{When}}~~0<\alpha<1,~\varpi_{2,\ell}^{(\alpha)}<0 \;\;\textit{and}\;\;\varpi_{2,\ell}^
 {(\alpha)}<\varpi_{2,\ell+1}^{(\alpha)} \;\;hold\;\;\textit{for}\;\;\ell\geq4,
\vspace{0.2 cm}\\
\displaystyle
(3)~{\textit{When}}~~1<\alpha<2,~\varpi_{2,\ell}^{(\alpha)}>0 \;\;\textit{and}\;\;\varpi_{2,\ell}^{(\alpha)}
>\varpi_{2,\ell+1}^{(\alpha)}\;\;hold\;\;\textit{for}\;\;\ell\geq5.
\end{array}
$$
\end{theorem}}
\begin{proof}
See the Appendix A.
\end{proof}

%Now we continue discuss the monotonicity of the other high-order coefficients.

Although it is not facile to prove the monotonicity of the coefficients
$\varpi_{p,\ell}^{(\alpha)},\;p=3,\cdots,6,$ we can explicitly
write their expressions, see the beginning part of section 2 for more details.
These coefficients are explicitly expressed which are beneficial for numerical
 calculations. Besides, through a huge of numerical simulations, one can find the
  monotonicity of the coefficients $\varpi_{p,\ell}^{(\alpha)},\;p=3,\cdots,5.$

Here, we only list very limited figures to show the monotonicity of the coefficients for $p=3,4,5$.
Figs. \ref{fig.1} and \ref{fig.2} show the monotonicity of the coefficients $\varpi_{3,\ell}^{(\alpha)}$ for $\alpha\in(0,1)$ and $\alpha\in(1,2)$.
Figs. \ref{fig.3} and \ref{fig.4} display the monotonicity of the coefficients $\varpi_{4,\ell}^{(\alpha)}$ for $\alpha\in(0,1)$ and $\alpha\in(1,2)$.
Figs. \ref{fig.5} and \ref{fig.6} present the monotonicity of the coefficients $\varpi_{5,\ell}^{(\alpha)}$ for $\alpha\in(0,1)$ and $\alpha\in(1,2)$. But through the numerical simulations, $\varpi_{6,\ell}^{(\alpha)}, \; \alpha\in(0,1)$ and $\alpha\in(1,2)$ seem not to have the monotonicity.

\begin{figure}[htbp]
\includegraphics[width=4.6in]{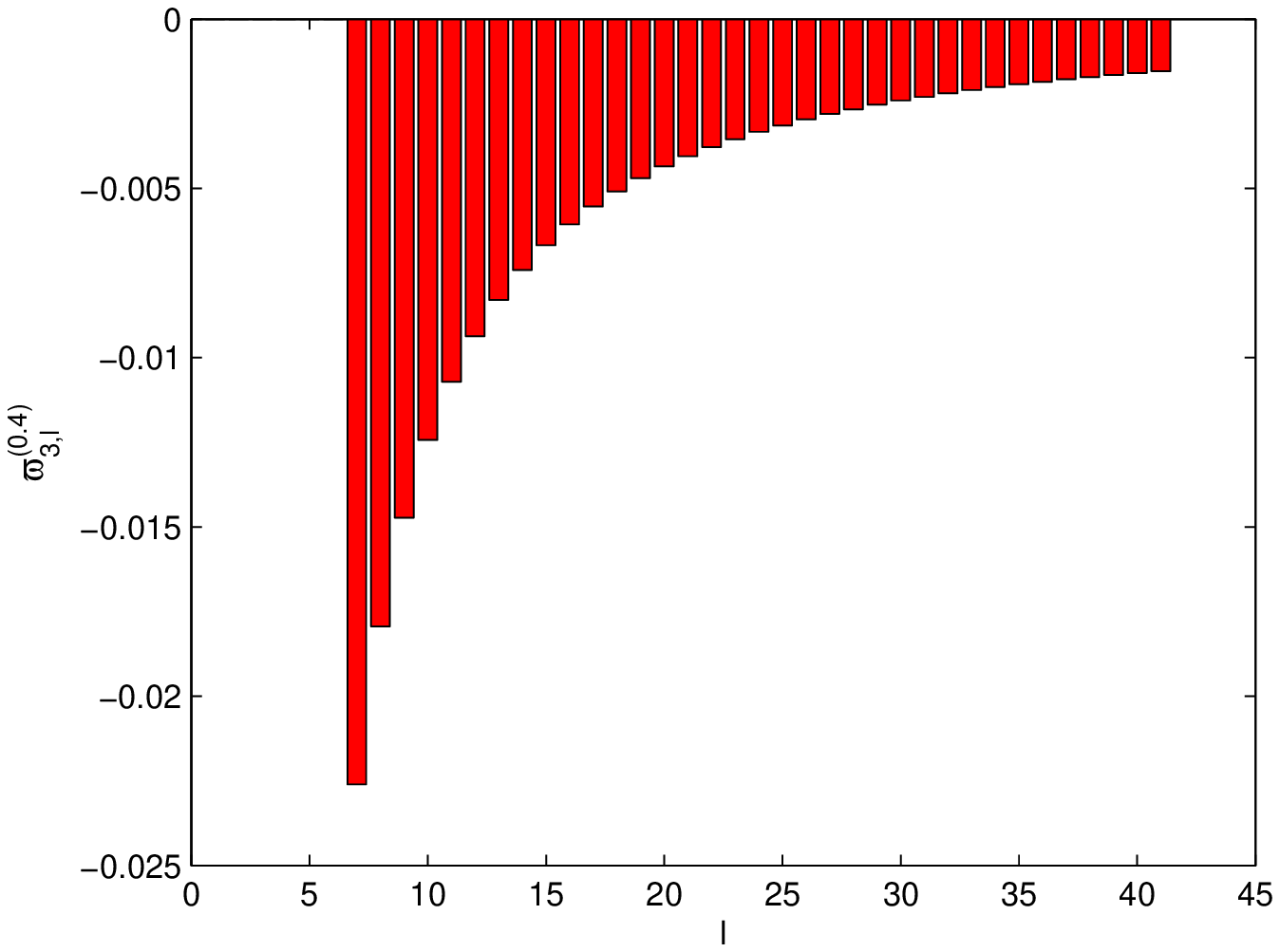}
\caption{The values of coefficient $\displaystyle\varpi_{3,\ell}^{(\alpha)}$
$(\ell=4,5,\cdots)$ for $\alpha=0.4$.}
\label{fig.1}
\end{figure}

\begin{figure}[htbp]
\includegraphics[width=4.6in]{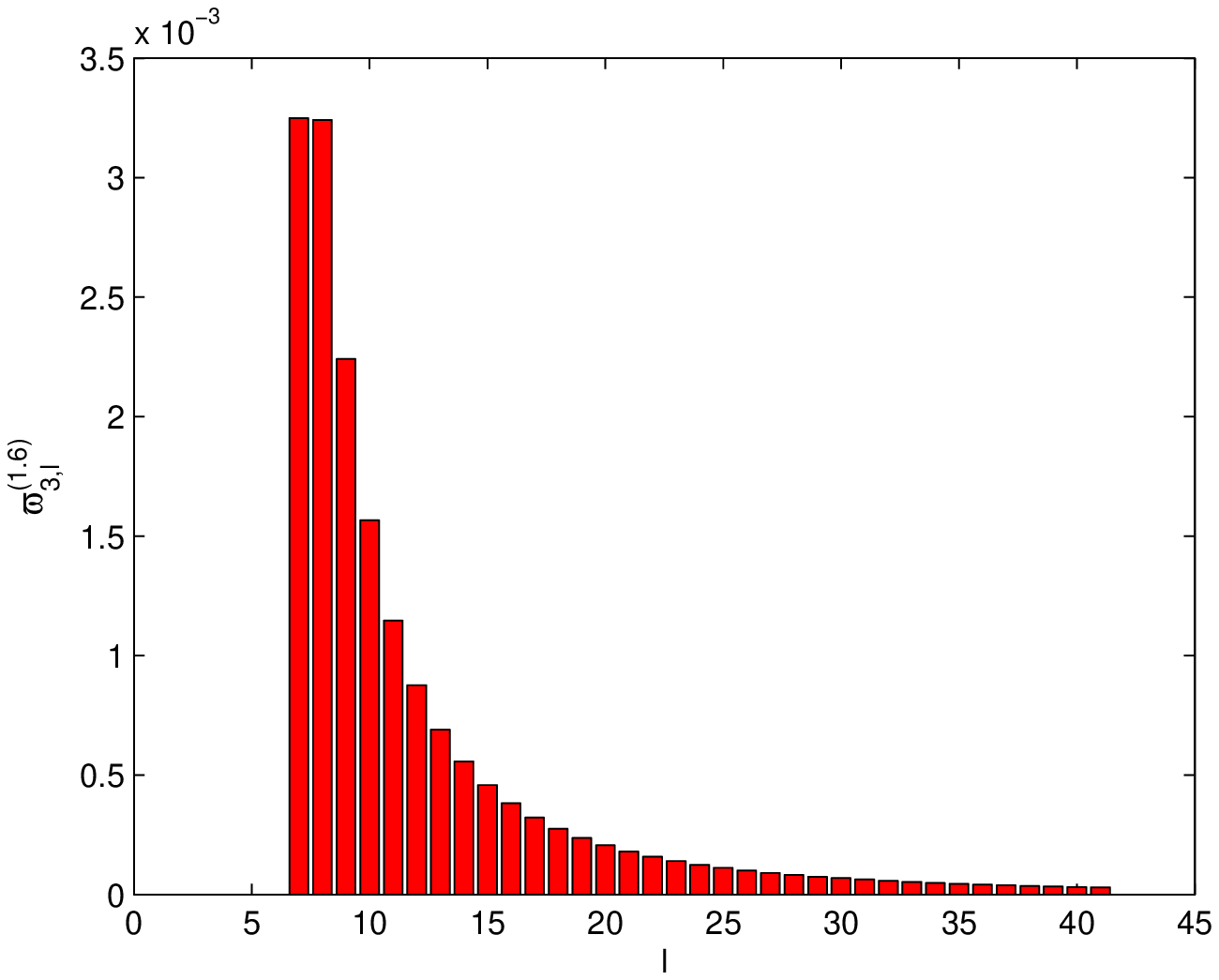}
\caption{The values of coefficient $\displaystyle\varpi_{3,\ell}^{(\alpha)}$
$(\ell=7,8,\cdots)$ for $\alpha=1.6$.}
\label{fig.2}
\end{figure}

\begin{figure}[htbp]
\includegraphics[width=4.6in]{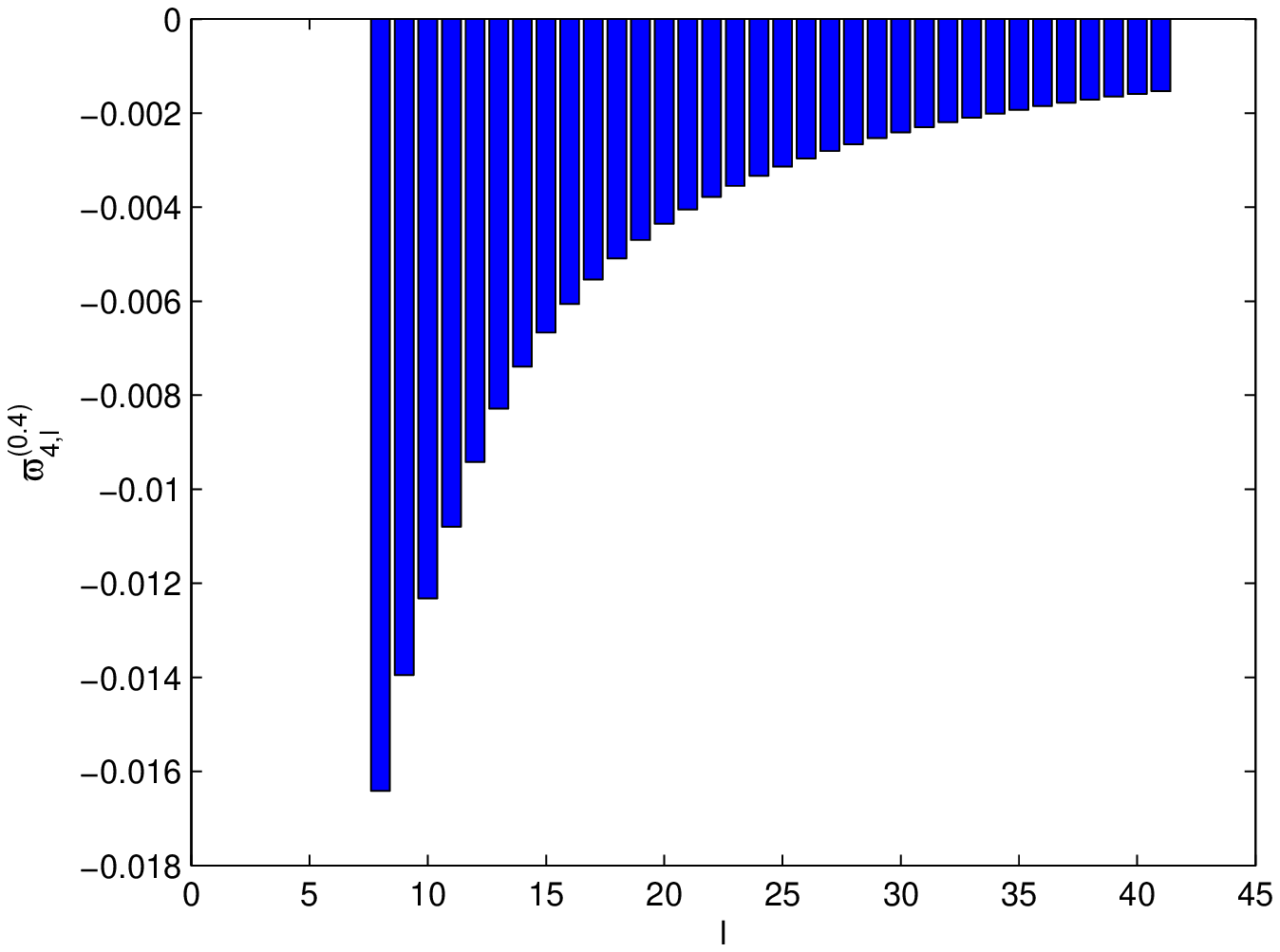}
\caption{The values of coefficient $\displaystyle\varpi_{4,\ell}^{(\alpha)}$
$(\ell=7,8,\cdots)$ for $\alpha=0.4$.}
\label{fig.3}
\end{figure}

\begin{figure}[htbp]
\includegraphics[width=4.6in]{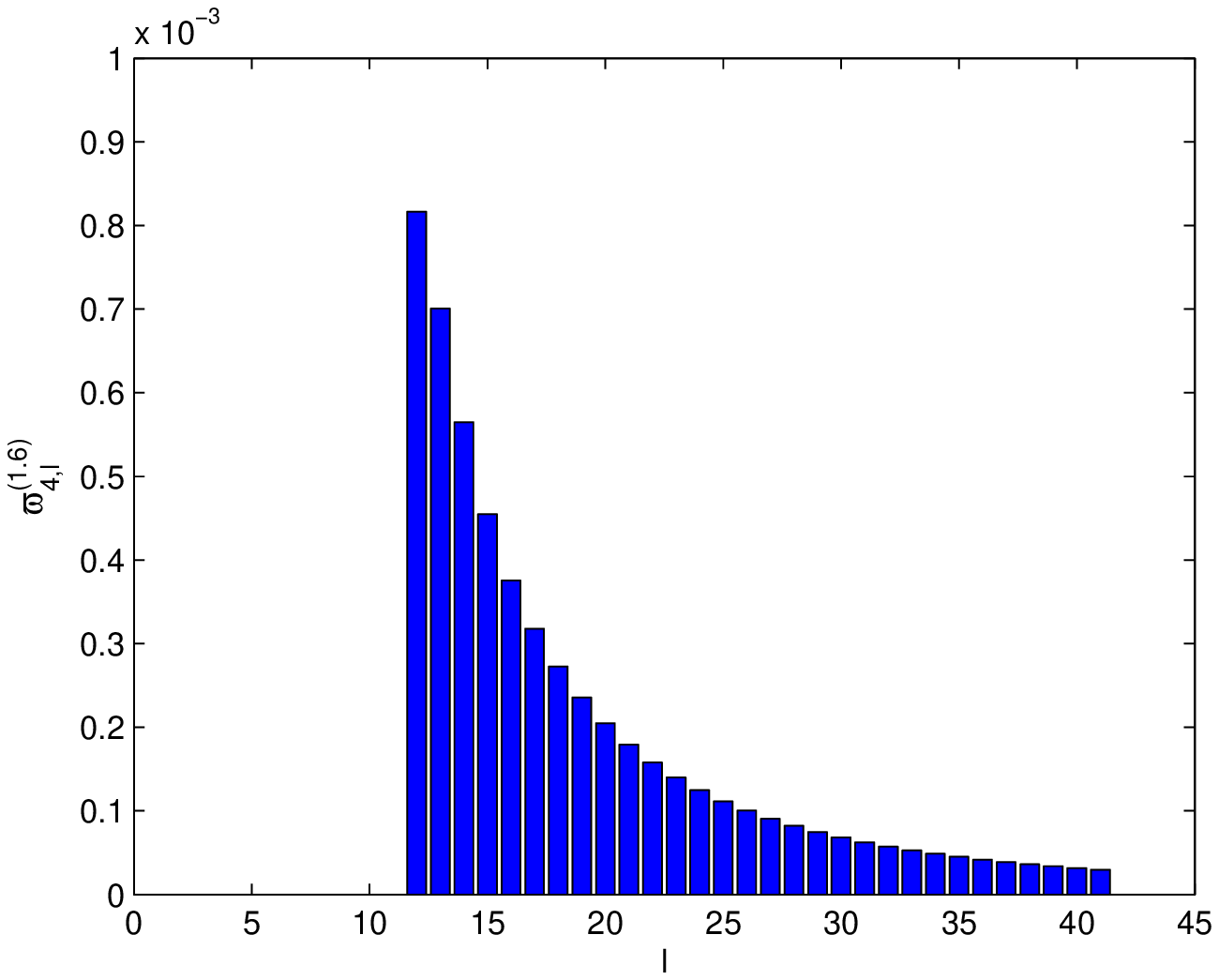}
\caption{The values of coefficient $\displaystyle\varpi_{4,\ell}^{(\alpha)}$
$(\ell=10,11,\cdots)$ for $\alpha=1.6$.}
\label{fig.4}
\end{figure}

\begin{figure}[htbp]
\includegraphics[width=4.6in]{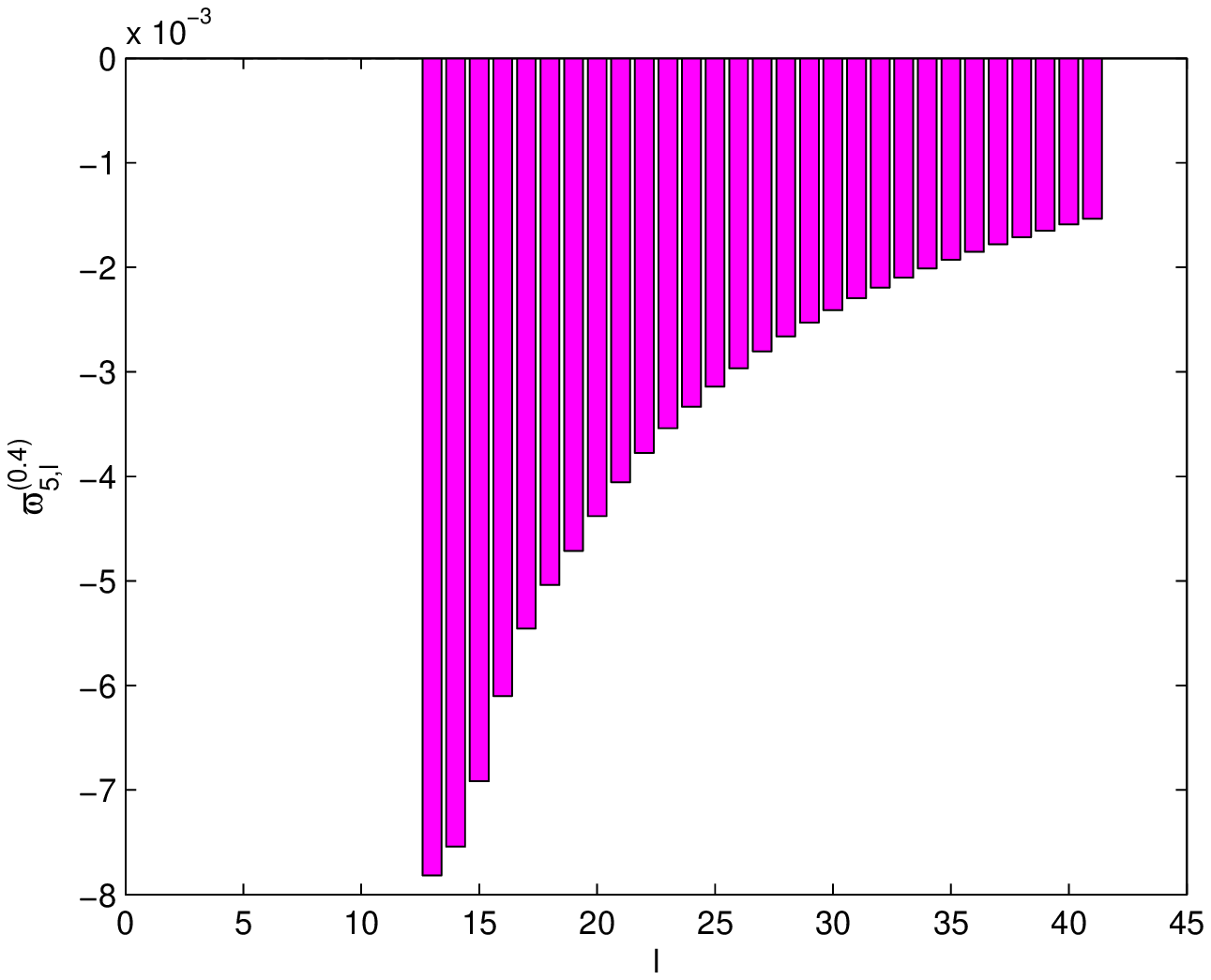}
\caption{The values of coefficient $\displaystyle\varpi_{5,\ell}^{(\alpha)}$
$(\ell=12,13,\cdots)$ for $\alpha=0.4$.}
\label{fig.5}
\end{figure}

\begin{figure}[htbp]
\includegraphics[width=4.6in]{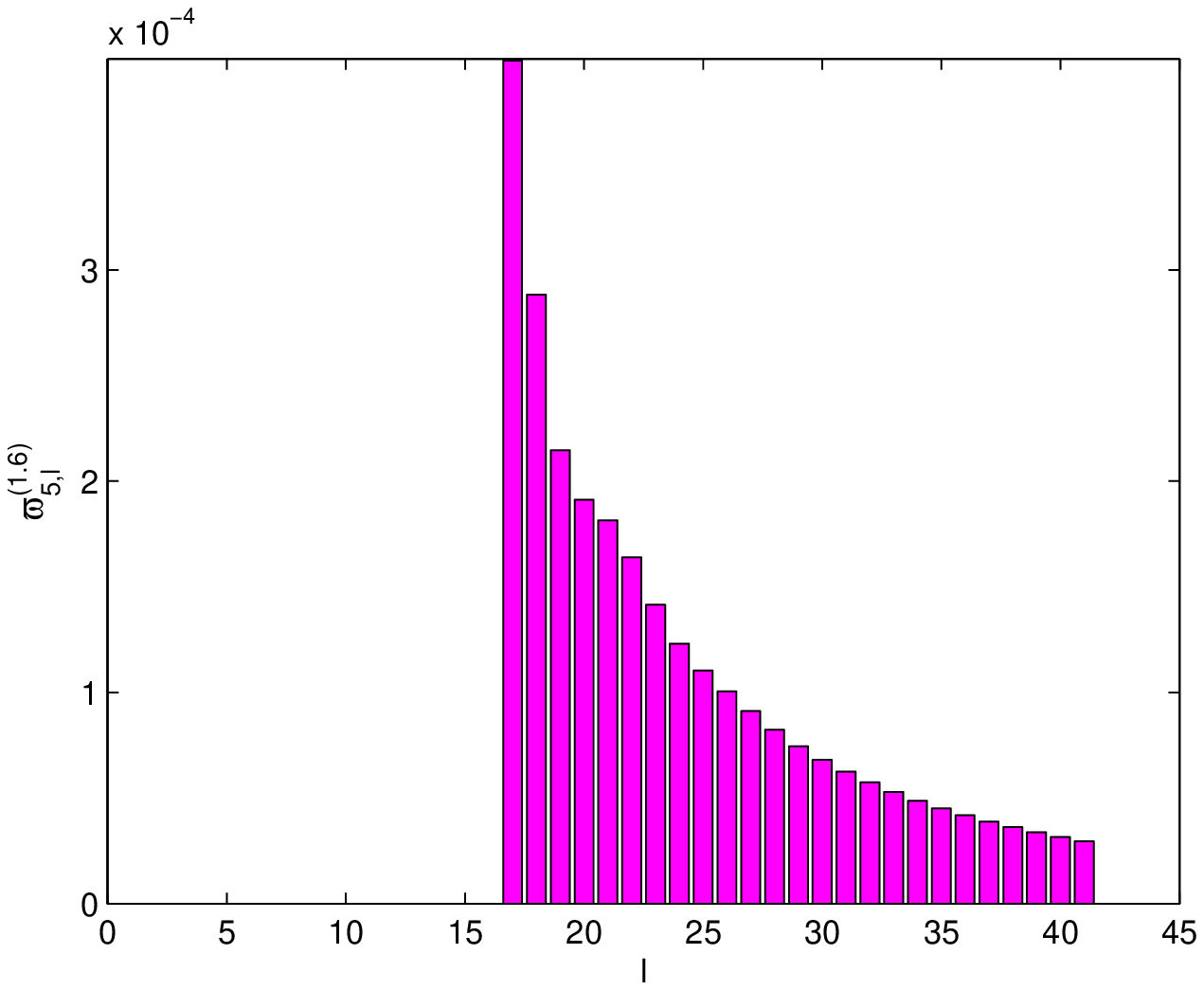}
\caption{The values of coefficient $\displaystyle\varpi_{5,\ell}^{(\alpha)}$
$(\ell=16,17,\cdots)$ for $\alpha=1.6$.}
\label{fig.6}
\end{figure}

%\begin{figure}[htbp]
%\begin{minipage}[t]{0.5\linewidth}
%\centering
%\includegraphics[width=1.0in]{F31}
%\caption{The values of coefficient $\displaystyle\varpi_{3,\ell}^{(\alpha)}$
%$(\ell=4,5,\cdots)$ for $\alpha=0.4$.}\label{fig.1}
%\end{minipage}
%\hspace{2ex}
%\begin{minipage}[t]{0.5\linewidth}
%\centering
%\includegraphics[width=1.0in]{F34}
%  \caption{The values of coefficient $\displaystyle\varpi_{3,\ell}^{(\alpha)}$
%$(\ell=7,8,\cdots)$ for $\alpha=1.6$.}\label{fig.2}
%%\label{fig:side:b1}
%\end{minipage}
%\end{figure}
%
%\begin{figure}[htbp]
%\begin{minipage}[t]{0.5\linewidth}
%\centering
%\includegraphics[width=3.0in]{F41}
%  \caption{The values of coefficient $\displaystyle\varpi_{4,\ell}^{(\alpha)}$
%$(\ell=7,8,\cdots)$ for $\alpha=0.4$.}\label{fig.3}
%\end{minipage}
%\hspace{2ex}
%\begin{minipage}[t]{0.5\linewidth}
%\centering
%\includegraphics[width=3.0in]{F44}
%  \caption{The values of coefficient $\displaystyle\varpi_{4,\ell}^{(\alpha)}$
%$(\ell=10,11,\cdots)$ for $\alpha=1.6$.}\label{fig.4}
%%\label{fig:side:b2}
%\end{minipage}
%\end{figure}
%%
%%
%%
%\begin{figure}[htbp]
%\begin{minipage}[t]{0.5\linewidth}
%\centering
%\includegraphics[width=3.0in]{F51}
%  \caption{The values of coefficient $\displaystyle\varpi_{5,\ell}^{(\alpha)}$
%$(\ell=12,13,\cdots)$ for $\alpha=0.4$.}\label{fig.5}
%\end{minipage}
%\hspace{2ex}
%\begin{minipage}[t]{0.5\linewidth}
%\centering
%\includegraphics[width=3.0in]{F54}
%  \caption{The values of coefficient $\displaystyle\varpi_{5,\ell}^{(\alpha)}$
%$(\ell=16,17,\cdots)$ for $\alpha=1.6$.}\label{fig.6}
%%\label{fig:side:b3}
%\end{minipage}
%\end{figure}

%{\color{red}
{In the following, we give the conjecture which is seemingly primary but hard to prove.\\
{\it{\bf Conjecture }
$$
\begin{array}{lll}
\displaystyle
(1)\;\;
\textrm{If}\;\;\displaystyle~0<\alpha<1,\;\;
then\;\;\varpi_{3,\ell}^{(\alpha)}\leq\varpi_{3,\ell+1}^{(\alpha)}\;\;holds
\;\;
for\;\;\ell\geq4,\;\;\varpi_{4,\ell}^{(\alpha)}\leq\varpi_{4,\ell+1}^{(\alpha)}\;\;holds\;\;\vspace{0.2 cm}\\\hspace{0.7cm}
for\;\;\ell\geq7,\;\;and \;\;\varpi_{5,\ell}^{(\alpha)}\leq\varpi_{5,\ell+1}^{(\alpha)}
\;\;
for\;\;\ell\geq12.
\vspace{0.2 cm}\\
(2)\;\;\textrm{If} \;\;\displaystyle~1<\alpha<2,\;\;then\;\;
\displaystyle\varpi_{3,\ell}^{(\alpha)}\geq\varpi_{3,\ell+1}^{(\alpha)}\;\;holds\;\;
for\;\;\ell\geq7,\;\;\varpi_{4,\ell}^{(\alpha)}\geq\varpi_{4,\ell+1}^{(\alpha)}\;\;holds\;\;\vspace{0.2 cm}\\\hspace{0.7cm}
for\;\;\ell\geq12,\;\;and \;\;\varpi_{5,\ell}^{(\alpha)}\geq\varpi_{5,\ell+1}^{(\alpha)}\;\;
for\;\;\ell\geq16.
\end{array}
$$}}

Besides their monotonicity, studying bounds of these coefficients is also of importance, which can be used to analyze the
stability and convergence for time fractional differential equations. In \cite{dim},
Dimitrov given the bounds for first-order $\varpi_{1,\ell}^{(\alpha)},(0<\alpha<1)$ below:

\begin{theorem}\label{th.4} The first-order coefficients
$\varpi_{1,\ell}^{(\alpha)},\;(0<\alpha<1)$ satisfy
$$
\begin{array}{lll}
\displaystyle
(1)\;\;\;\widetilde{B}_{1}^{L}{(\alpha,\ell)}
<\left|\varpi_{1,\ell}^{(\alpha)}\right|<B_{1}^{R}{(\alpha,\ell)},\;\;\;
\ell\geq3,
\vspace{0.2 cm}\\\displaystyle\;\;where\;\;\widetilde{B}_{1}^{L}{(\alpha,\ell)}
=\exp\left(-(\alpha+1)^2\left(\frac{\pi^2}{6}-\frac{5}{4}\right)\right)\frac{\alpha(1-\alpha)2^{\alpha}}{\ell^{\alpha+1}},\;
B_{1}^{R}{(\alpha,\ell)}=\frac{\alpha2^{\alpha+1}}{(\ell+1)^{\alpha+1}},\vspace{0.2 cm}\\
\displaystyle(2)\;\;\; \widetilde{S}_{1}^{L}{(\alpha,\ell)}
<\sum\limits_{k=\ell}^{\infty}\left|\varpi_{1,k}^{(\alpha)}\right|<S_{1}^{R}{(\alpha,\ell)},\;\;\ell\geq3,
\vspace{0.2 cm}\\ \displaystyle\;\;where\;\;\widetilde{S}_{1}^{L}{(\alpha,\ell)}
=\frac{1-\alpha}{5}\left(\frac{2}{\ell}\right)^{\alpha},\;
S_{1}^{R}{(\alpha,\ell)}=2\left(\frac{2}{\ell}\right)^{\alpha}.
\end{array}
$$
\end{theorem}

In this paper, we can give tighter estimates for the lower bounds. See the following theorem.

\begin{theorem}\label{th.5} The first-order coefficients
$\varpi_{1,\ell}^{(\alpha)},\;(0<\alpha<1)$ satisfy
$$
\begin{array}{lll}
\displaystyle
(1)\;\;\;B_{1}^{L}{(\alpha,\ell)}
<\left|\varpi_{1,\ell}^{(\alpha)}\right|<B_{1}^{R}{(\alpha,\ell)},\;\ell\geq3,\;where\;\;B_{1}^{L}{(\alpha,\ell)}
=\frac{\alpha(1-\alpha)}{2}\left(\frac{2}{\ell}\right)^{2(\alpha+1)},\vspace{0.2 cm}\\
\displaystyle(2)\;\;\; S_{1}^{L}{(\alpha,\ell)}
<\sum\limits_{k=\ell}^{\infty}\left|\varpi_{1,k}^{(\alpha)}\right|<S_{1}^{R}{(\alpha,\ell)},
\;\ell\geq3,\;where\;\;S_{1}^{L}{(\alpha,\ell)}
=\frac{\alpha(1-\alpha)}{2\alpha+1}\left(\frac{2}{\ell}\right)^{2\alpha+1}.
\end{array}
$$
\end{theorem}
%\begin{proof} See the Appendix B.\end{proof}

Next, we list the comparison theorem for $\widetilde{B}_{1}^{L}{(\alpha,\ell)}$ and ${B}_{1}^{L}{(\alpha,\ell)}$,
 $\widetilde{S}_{1}^{L}{(\alpha,\ell)}$ and ${S}_{1}^{L}{(\alpha,\ell)}$, respectively.

\begin{theorem}\label{th.6} The following inequalities hold:
$$
\begin{array}{lll}\displaystyle
(1)\;\;{B}_{1}^{L}(\alpha,3)<\widetilde{B}_{1}^{L}(\alpha,3)\;\;for\;\;0<\alpha<\frac{12\ln\frac{3}{2}}{2\pi^2-15}-1\approx0.0267;
\;\;{B}_{1}^{L}(\alpha,4)<\widetilde{B}_{1}^{L}(\alpha,4)\;\;for\\\displaystyle\;\;0<\alpha<\frac{12\ln2}{2\pi^2-15}-1\approx0.7551;
\;\;{B}_{1}^{L}(\alpha,\ell)<\widetilde{B}_{1}^{L}(\alpha,\ell)\;\;for\;\;0<\alpha<1\;\;and\;\;\ell\geq5;\vspace{0.2cm}
\\
(2)\;\;{S}_{1}^{L}(\alpha,\ell)<\widetilde{S}_{1}^{L}(\alpha,\ell)\;\;for\;\;0<\alpha<1\;\;and\;\;\ell\geq3.
\end{array}
$$
\end{theorem}
%\begin{proof} See the Appendix B.\end{proof}

Finally, we give the bounds for the other coefficients.

\begin{theorem}\label{th.7} The first-order coefficients
$\varpi_{1,\ell}^{(1+\alpha)},\;(0<\alpha<1)$ satisfy
$$
\begin{array}{lll}
(1)\;
\displaystyle{\overline{B}}_{1}^{L}{(1+\alpha,\ell)}
<\left|\varpi_{1,\ell}^{(1+\alpha)}\right|<{\overline{B}}_{1}^{R}{(1+\alpha,\ell)},\;\ell\geq4,\vspace{0.2 cm}\\\displaystyle\;where\;\;
{\overline{B}}_{1}^{L}{(1+\alpha,\ell)}=\frac{(1-\alpha)\alpha(1+\alpha)}{6}\left(\frac{3}{\ell}\right)^{2(2+\alpha)},
{\overline{B}}_{1}^{R}{(1+\alpha,\ell)}=\frac{\alpha(1+\alpha)}{2}\left(\frac{3}{\ell+1}\right)^{2+\alpha}
,\vspace{0.2 cm}\\
(2)\;\displaystyle
{\overline{S}}_{1}^{L}{(1+\alpha,\ell)}
<\sum\limits_{k=\ell}^{\infty}\left|\varpi_{1,k}^{(1+\alpha)}\right|
<{\overline{S}}_{1}^{R}{(1+\alpha,\ell)},\;\ell\geq4,\vspace{0.2 cm}\\\displaystyle\;where\;\;
{\overline{S}}_{1}^{L}{(1+\alpha,\ell)}=\frac{(1-\alpha)\alpha(1+\alpha)}{2(3+2\alpha)}\left(\frac{3}{\ell}\right)^{3+2\alpha},\;
{\overline{S}}_{1}^{R}{(1+\alpha,\ell)}=
\frac{3\alpha}{2}\left(\frac{3}{\ell}\right)^{1+\alpha}.
\end{array}
$$
\end{theorem}
%
%\begin{proof} See the Appendix B.\end{proof}

\begin{theorem}\label{th.8} The second-order coefficients
$\varpi_{2,\ell}^{(\alpha)}$ and $\varpi_{2,\ell}^{(\alpha+1)}$ for $0<\alpha<1$ satisfy
$$
\begin{array}{lll}
(1)\;\;\;\displaystyle{B}_{2}^{L}{(\alpha,\ell)}
<\left|\varpi_{2,\ell}^{(\alpha)}\right|<{B}_{2}^{R}{(\alpha,\ell)},\;\ell\geq4,\;\vspace{0.2 cm}\\\;\;where\;\;
{B}_{2}^{L}{(\alpha,\ell)}=
\displaystyle\left(\frac{3}{2}\right)^{\alpha}\left[\left(1+\left(\frac{1}{3}\right)^{\ell}\right)
\frac{\alpha(1-\alpha)}{2}\left(\frac{2}{\ell}\right)^{2(1+\alpha)}-\left(1-\left(\frac{1}{3}\right)^{\ell-1}\right)
\frac{\alpha^22^{2\alpha+1}}{1+(\alpha+1)\ell}\right],\vspace{0.3 cm}\\
{B}_{2}^{R}{(\alpha,\ell)}=\displaystyle\left(\frac{3}{2}\right)^{\alpha}\left[\left(1+\left(\frac{1}{3}\right)^{\ell}\right)
\frac{\alpha2^{\alpha+1}}{(\ell+1)^{\alpha+1}}-
\frac{\alpha^2(1-\alpha)^24^{2\alpha+1}}{2}\left(1-\left(\frac{1}{3}\right)^{\ell-1}\right)
\left(\frac{2}{\ell}\right)^{4(\alpha+1)}\right].
\end{array}
$$
$$
\begin{array}{lll}
(2)\;\;\;\displaystyle{\overline{B}}_{2}^{L}{(1+\alpha,\ell)}
<\left|\varpi_{2,\ell}^{(1+\alpha)}\right|<{\overline{B}}_{2}^{R}{(1+\alpha,\ell)},\;\ell\geq4,\vspace{0.2 cm}\\\;\;where\;\;
\displaystyle{\overline{B}}_{2}^{L}{(1+\alpha,\ell)}=\left(\frac{3}{2}\right)^{1+\alpha}
\left[\left(1+\left(\frac{1}{3}\right)^{\ell}\right)
\frac{(1-\alpha)\alpha(1+\alpha)}{6}\left(\frac{3}{\ell}\right)^{2(2+\alpha)}\right.\vspace{0.2cm}\\
+\displaystyle \frac{(1-\alpha)^2\alpha^2(1+\alpha)^2}{216}\left(1-\left(\frac{1}{3}\right)^{\ell-3}\right)
\left(\frac{6}{\ell}\right)^{4(2+\alpha)}
-\left.\frac{\alpha(1+\alpha)^2}{2}\left(\frac{1}{3} +\left(\frac{1}{3}\right)^{\ell-1}\right)
\left(\frac{3}{\ell}\right)^{2+\alpha} \right],\vspace{0.2cm}\\

\displaystyle{\overline{B}}_{2}^{R}{(1+\alpha,\ell)}=\left(\frac{3}{2}\right)^{1+\alpha}
\left[\left(1+\left(\frac{1}{3}\right)^{\ell}\right)
\frac{\alpha(\alpha+1)3^{\alpha+2}}{2(\ell+1)^{\alpha+2}} +
\left(1-\left(\frac{1}{3}\right)^{\ell-3}\right)
\frac{\alpha^2(1+\alpha)^23^{2(2+\alpha)}}{24(1+(2+\alpha)\ell)}
\right.\vspace{0.2 cm}\\\displaystyle\left.-\frac{(1-\alpha)\alpha(1+\alpha)^2}{6}
\left(\frac{1}{3}
+\left(\frac{1}{3}\right)^{\ell-1}\right)\left(\frac{3}{\ell-1}\right)^{2(2+\alpha)}
\right].
\end{array}$$

\end{theorem}
%
%\begin{proof} See the Appendix C.\end{proof}
 All of the proofs for theorems \ref{th.4}-\ref{th.8} are given in Appendix B.

\section{Numerical methods for the Riesz-type turbulent diffusion
equation}

Define $t_k=k\tau$, $k =
0, 1,\ldots,N$, and $\tau=\frac{T}{N}$ for a given $T>0$, $h =\frac{b-a}{M}$ is the
equidistant grid size in space, $x_j = a + jh$, $j = 0, 1, \ldots
,M$.
\subsection{The 2nd-order scheme in space}
Firstly, using the Crank-Nicolson method for the Riesz space fractional turbulent diffusion equation
(\ref{eq.1}) in time direction, we can obtain
\begin{eqnarray}
\begin{array}{lll}
\displaystyle \frac{
u(x_j,t_{k+1})-u(x_j,t_{k})}{\tau}=\frac{1}{2}\left(d_2\frac{\partial^2
u(x_j,t_{k+1})}{\partial x^2}-d_1\frac{\partial
u(x_j,t_{k+1})}{\partial x}\right.\vspace{0.2 cm}\\
\displaystyle +d_{\alpha}\frac{\partial^\alpha
u(x_j,t_{k+1}) }{\partial |x|^{\alpha}} +d_2\frac{\partial^2 u(x_j,t_{k})}{\partial
x^2}-d_1\frac{\partial u(x_j,t_{k}) }{\partial
x}\vspace{0.2 cm}\\ \displaystyle\left.
+d_{\alpha}\frac{\partial^\alpha u(x_j,t_{k})}{\partial
|x|^{\alpha}}\right)+s(x_j,t_{k+\frac{1}{2}})+\mathcal
{O}(\tau^2).\label{eq.6}
\end{array}
\end{eqnarray}

Secondly, for the first- and second-order derivatives, we using the
following approximations, respectively
\begin{equation}
 \displaystyle\frac{\partial u(x_j,t_k)}{\partial
 x}={\mu_x\delta_{{x}}u(x_j,t_k)}+\mathcal {O}(h^2),\;\label{eq.7}
\end{equation}
and
\begin{equation}
 \displaystyle\frac{\partial^2 u(x_j,t_k)}{\partial
 x^2}={\delta_x^2u(x_j,t_k)}+\mathcal {O}(h^2),\label{eq.8}
\end{equation}
where $\mu_x\delta_{{x}}$ and $\delta_x^2$ are defined
by
$
\mu_x\delta_{{x}}u(x_j,t_k)=\frac{u(x_{j+1},t_k)-u(x_{j-1},t_k)}{2h},
$
and
$
\delta_x^2u(x_j,t_k)=\frac{u(x_{j+1},t_k)-2u(x_j,t_k)+u(x_{j-1},t_k)}{h^2}.
$.

Next, we choose the second-order formula to
approximate Riesz derivative,
\begin{eqnarray}
\begin{array}{lll}\displaystyle
 \frac{\partial^\alpha u(x,t)}{\partial{|x|^\alpha}}&=&\displaystyle
-\frac{1}{2\cos\left(\pi\alpha/2\right)h^\alpha}\left
(\sum\limits_{\ell=0}^{\infty}\varpi_{2,\ell}^{(\alpha)}u(x-\ell
h,t)\right.\vspace{0.2 cm}\\&&\displaystyle\left.+\sum\limits_{\ell=0}^{\infty}\varpi_{2,\ell}^{(\alpha)}u(x+\ell
h,t)\right)+\mathcal {O}(h^2).\label{eq.9}
\end{array}
\end{eqnarray}

Substituting (\ref{eq.7}), (\ref{eq.8}) and (\ref{eq.9}) into (\ref{eq.6}) and removing the high-order
term yield
\begin{eqnarray}
\begin{array}{lll}
\displaystyle
\left(\frac{d_2}{h^2}+\frac{d_1}{2h}\right)u_{j-1}^{k+1}-\left(\frac{2}{\tau}+\frac{2d_2}{h^2}\right)u_{j}^{k+1}+
\left(\frac{d_2}{h^2}-\frac{d_1}{2h}\right)u_{j+1}^{k+1}\vspace{0.1 cm}\\
\displaystyle
-\nu\left[\sum\limits_{\ell=0}^{j}\varpi_{2,\ell}^{(\alpha)}
u_{j-l}^{k+1} +\sum\limits_{\ell=0}^{M-j}\varpi_{2,\ell}^{(\alpha)}
u_{j+l}^{k+1} \right]=-
\left(\frac{d_2}{h^2}+\frac{d_1}{2h}\right)u_{j-1}^{k}\vspace{0.1 cm}\\
\displaystyle-\left(\frac{2}{\tau}-\frac{2d_2}{h^2}\right)u_{j}^{k}-
\left(\frac{d_2}{h^2}-\frac{d_1}{2h}\right)u_{j+1}^{k}
\displaystyle
+\nu\left[\sum\limits_{\ell=0}^{j}\varpi_{2,\ell}^{(\alpha)}
u_{j-l}^{k} \right.\vspace{0.1 cm}\\\displaystyle\left.+\sum\limits_{\ell=0}^{M-j}\varpi_{2,\ell}^{(\alpha)}
u_{j+l}^{k}
\right]-2s_j^{k+\frac{1}{2}},
\displaystyle\;\;j=1,\ldots,M-1,\;k=0,1,\ldots,N-1,\label{eq.10}
\end{array}
\end{eqnarray}
where $\displaystyle
\nu=\frac{d_{\alpha}}{2\cos\left(\pi\alpha/2\right)h^\alpha}. $

 Next we discuss the stability and convergence of scheme (\ref{eq.10}).

\begin{theorem} The numerical scheme (\ref{eq.10}) is
unconditionally stable and convergent with order $\mathcal O(\tau^2+h^2)$.
\end{theorem}

 \begin{proof}
 Assume that the solution of equation (\ref{eq.1})
can be zero extended to the whole real line $R$. Suppose $U_j^k$ is
the approximation solution of equation (\ref{eq.10}) and let $\mathcal
{E}_j^k=U_j^k-u_j^k$, then the error
equation reads as
\begin{eqnarray}
\begin{array}{lll}
\displaystyle \left(\frac{d_2}{h^2}+\frac{d_1}{2h}\right)\mathcal
{E}_{j-1}^{k+1}-\left(\frac{2}{\tau}+\frac{2d_2}{h^2}\right)\mathcal
{E}_{j}^{k+1}+
\left(\frac{d_2}{h^2}-\frac{d_1}{2h}\right)\mathcal {E}_{j+1}^{k+1}\vspace{0.2 cm}\\
\displaystyle
-\nu\left[\sum\limits_{\ell=0}^{\infty}\varpi_{2,\ell}^{(\gamma)}
\mathcal {E}_{j-l}^{k+1}
+\sum\limits_{\ell=0}^{\infty}\varpi_{2,\ell}^{(\gamma)} \mathcal
{E}_{j+l}^{k+1} \right]=-
\left(\frac{d_2}{h^2}+\frac{d_1}{2h}\right)\mathcal {E}_{j-1}^{k}\vspace{0.2 cm}\\
\displaystyle-\left(\frac{2}{\tau}-\frac{2d_2}{h^2}\right)\mathcal
{E}_{j}^{k}- \left(\frac{d_2}{h^2}-\frac{d_1}{2h}\right)\mathcal
{E}_{j+1}^{k}
+\nu\left[\sum\limits_{\ell=0}^{\infty}\varpi_{2,\ell}^{(\gamma)}
\mathcal {E}_{j-l}^{k}\right.\vspace{0.2 cm}\\\displaystyle\left.
+\sum\limits_{\ell=0}^{\infty}\varpi_{2,\ell}^{(\gamma)} \mathcal
{E}_{j+l}^{k}
\right],
j=1,\ldots,M-1,\;k=0,1,\ldots,N-1.\label{eq.11}
\end{array}
\end{eqnarray}

Let $\mathcal
{E}_j^k=\xi^k e^{ij\theta}$ be the
solution of equation (\ref{eq.11}), $i=\sqrt{-1},$  where $\theta\in[-\pi,\pi]$ is
called the phase angle. The
stability condition of scheme (\ref{eq.10}) is $|\xi(\theta)|\leq 1$
for all $\theta\in[-\pi,\pi]$.

Substituting $\varepsilon_j^k=\xi^k e^{ij\theta}$ into (\ref{eq.11}) gives
$$
\begin{array}{lll}
\displaystyle
\xi(\theta)=\frac{\left(\frac{2h}{\tau}-\frac{4d_2}{h}\sin^2\left(\frac{\theta}{2}\right)
-2\nu
h\sum\limits_{\ell=0}^{\infty}\varpi_{2,\ell}^{(\gamma)}\cos(\ell\theta)\right)-i
d_1\sin(\theta)}
{\left(\frac{2h}{\tau}+\frac{4d_2}{h}\sin^2\left(\frac{\theta}{2}\right)
+2\nu
h\sum\limits_{\ell=0}^{\infty}\varpi_{2,\ell}^{(\gamma)}\cos(\ell\theta)\right)+i
d_1\sin(\theta)}.
\end{array}
$$

According to Theorem \ref{th.2}, we easily obtain
$\displaystyle \left|\xi(\theta)\right|\leq1.$
So scheme (\ref{eq.10}) is unconditionally stable. It is easy to show the convergence order of
scheme (\ref{eq.10}) is $\mathcal O(\tau^2+h^2)$. In effect, the proof is almost the same as that of \cite{liding2}.
\end{proof}

\subsection{The 4th-order scheme in space}
Let us first consider the following differential equation
\begin{eqnarray}\displaystyle
d_2\frac{\partial^2u(x,t)}{\partial x^2} -d_1\frac{\partial
u(x,t)}{\partial x}=g(x,t).\label{eq.12}
\end{eqnarray}

Applying the technique presented in \cite{moh}, we can obtain a
fourth-order difference scheme for solving the above equation
\begin{eqnarray}\displaystyle
\Delta_4u(x_j,t)=\widetilde{\Delta}_4g(x_j,t)+\mathcal {O}(h^4).\label{eq.13}
\end{eqnarray}
where
$$
\begin{array}{lll}\displaystyle
\Delta_4=\left(d_2+\frac{d_1^2h^2}{12d_2}\right)\delta_x^2-d_1\mu_x\delta_{{x}},\;\;
\widetilde{\Delta}_4=I+\frac{h^2}{12}\left(\delta_x^2-\frac{d_1}{d_2}\mu_x\delta_{{x}}\right),
\end{array}
$$
in which $I$ is a unit operator.

Combing (\ref{eq.6}) with (\ref{eq.12}) and (\ref{eq.13}) leads to
\begin{eqnarray}
\begin{array}{lll}
&\displaystyle
 \frac{ 2}{\tau}\widetilde{\Delta}_4\cdot\left(u(x_j,t_{k+1})-u(x_j,t_{k})\right)=
 \Delta_4\cdot\left(u(x_j,t_{k+1})+u(x_j,t_{k})\right) \vspace{0.2cm}\\& \displaystyle
 +d_{\alpha}\widetilde{\Delta}_4\cdot\left(\frac{\partial^{\alpha}
u(x_j,t_{k+1})}{\partial |x|^{\alpha}}+\frac{\partial
u(x_j,t_{k})}{\partial |x|^{\alpha}}\right)+2\widetilde{\Delta}_4\cdot
s(x_j,t_{k+\frac{1}{2}})\vspace{0.2 cm}\\&+ \mathcal {O}(\tau^2+h^4).\label{eq.14}
\end{array}
\end{eqnarray}

For the Riesz derivative in equation (\ref{eq.14}), we use the
fourth-order numerical scheme
$$
\begin{array}{lll}
\displaystyle \frac{\partial^\alpha u(x,t)}{\partial{|x|^\alpha}}=
-\frac{1}{2\cos\left(\pi\alpha/2\right)h^\alpha}\left
(\sum\limits_{\ell=0}^{\infty}\varpi_{4,\ell}^{(\alpha)}u(x-\ell
h,t)+\sum\limits_{\ell=0}^{\infty}\varpi_{4,\ell}^{(\alpha)}u(x+\ell
h,t)\right)+\mathcal {O}(h^4).
\end{array}
$$
Substituting it into (\ref{eq.14}) and ignoring the truncation error
term give
\begin{eqnarray}
\begin{array}{lll}
\displaystyle
\left(\frac{2a_1}{\tau}-b_1\right)u_{j-1}^{k+1}+\left(\frac{2a_2}{\tau}-b_2\right)u_{j}^{k+1}+
\left(\frac{2a_3}{\tau}-b_3\right)u_{j+1}^{k+1}\vspace{0.2 cm}\\
\displaystyle
+\nu\left[\sum\limits_{\ell=0}^{j}\varpi_{4,\ell}^{(\alpha)}
\left(a_1u_{j-l-1}^{k+1}+a_2u_{j-l}^{k+1}+a_3u_{j-l+1}^{k+1}\right)\right.\vspace{0.2 cm}\\\displaystyle\left.
+\sum\limits_{\ell=0}^{M-j}\varpi_{4,\ell}^{(\alpha)}
\left(a_1u_{j+l-1}^{k+1}+a_2u_{j+l}^{k+1}+a_3u_{j+l+1}^{k+1}\right)
\right]\vspace{0.2 cm}
\\\displaystyle=
\left(\frac{2a_1}{\tau}+b_1\right)u_{j-1}^{k}+\left(\frac{2a_2}{\tau}+b_2\right)u_{j}^{k}+
\left(\frac{2a_3}{\tau}+b_3\right)u_{j+1}^{k}
\vspace{0.2 cm}\\\displaystyle
-\nu\left[\sum\limits_{\ell=0}^{j}\varpi_{4,\ell}^{(\alpha)}
\left(a_1u_{j-l-1}^{k}+a_2u_{j-l}^{k}+a_3u_{j-l+1}^{k}\right)\right.\vspace{0.2 cm}\\\left.\displaystyle
+\sum\limits_{\ell=0}^{M-j}\varpi_{4,\ell}^{(\alpha)}
\left(a_1u_{j+l-1}^{k}+a_2u_{j+l}^{k}+a_3u_{j+l+1}^{k}\right)
\right]\vspace{0.2 cm}\\+2a_1s_{j-1}^{k+\frac{1}{2}}+2a_2s_{j}^{k+\frac{1}{2}}+2a_3s_{j+1}^{k+\frac{1}{2}},
j=1,\ldots,M-1,\vspace{0.2 cm}\\\;k=0,1,\ldots,N-1.\label{eq.15}
\end{array}
\end{eqnarray}
Here, the parameters in (\ref{eq.15}) are given below,
$$
\begin{array}{lll}
\displaystyle
\nu=\frac{d_{\alpha}}{2\cos\left(\pi\alpha/2\right)h^\alpha},\;\;
a_1=\frac{1}{12}+\frac{d_1h}{24d_2},\;\; a_2=\frac{5}{6},\;\;
a_3=\frac{1}{12}-\frac{d_1h}{24d_2},\vspace{0.2 cm}\\ \displaystyle
b_1=\frac{d_2}{h^2}+\frac{d_1^2}{12d_2}+\frac{d_1}{2h},\,\,
b_2=-2\left(\frac{d_2}{h^2}+\frac{d_1^2}{12d_2}\right),\;\;
b_3=\frac{d_2}{h^2}+\frac{d_1^2}{12d_2}-\frac{d_1}{2h}.
\end{array}
$$

%\begin{theorem}\label{th.5} When
%$\displaystyle 0<\gamma\leq 0.8439$, the difference scheme (15) is
%unconditionally stable.
%\end{theorem}
\begin{theorem} When
$\displaystyle 0<\alpha\leq 0.8439$, the numerical scheme (\ref{eq.15}) is
unconditionally stable and convergent with order $\mathcal O(\tau^2+h^4)$.
\end{theorem}

\begin{proof}
Similar to the above subsection, we can
get the error equation of equation
(\ref{eq.15})
%%below
%%\begin{eqnarray}
%%\begin{array}{lll}
%%\displaystyle \left(\frac{2a_1}{\tau}-b_1\right)\mathcal
%%{E}_{j-1}^{k+1}+\left(\frac{2a_2}{\tau}-b_2\right)\mathcal
%%{E}_{j}^{k+1}+
%%\left(\frac{2a_3}{\tau}-b_3\right)\mathcal {E}_{j+1}^{k+1}\\
%%\displaystyle
%%+\nu\left[\sum\limits_{\ell=0}^{j}\varpi_{4,\ell}^{(\alpha)}
%%\left(a_1\mathcal {E}_{j-l-1}^{k+1}+a_2\mathcal
%%{E}_{j-l}^{k+1}+a_3\mathcal
%%{E}_{j-l+1}^{k+1}\right)
%%+\sum\limits_{\ell=0}^{M-j}\varpi_{4,\ell}^{(\alpha)}
%%\left(a_1\mathcal {E}_{j+l-1}^{k+1}\right.\right.\\\left.\left.+a_2\mathcal
%%{E}_{j+l}^{k+1}+a_3\mathcal {E}_{j+l+1}^{k+1}\right)
%%\right]
%%= \left(\frac{2a_1}{\tau}+b_1\right)\mathcal
%%{E}_{j-1}^{k}+\left(\frac{2a_2}{\tau}+b_2\right)\mathcal
%%{E}_{j}^{k}\\\displaystyle+
%%\left(\frac{2a_3}{\tau}+b_3\right)\mathcal {E}_{j+1}^{k}
%%\displaystyle
%%-\nu\left[\sum\limits_{\ell=0}^{j}\varpi_{4,\ell}^{(\alpha)}
%%\left(a_1\mathcal {E}_{j-l-1}^{k}+a_2\mathcal
%%{E}_{j-l}^{k}+a_3\mathcal {E}_{j-l+1}^{k}\right)\right.\\\left.\displaystyle
%%+\sum\limits_{\ell=0}^{M-j}\varpi_{4,\ell}^{(\alpha)}
%%\left(a_1\mathcal {E}_{j+l-1}^{k}+a_2\mathcal
%%{E}_{j+l}^{k}+a_3\mathcal {E}_{j+l+1}^{k}\right)\right].
%%\label{eq.16}
%%\end{array}
%%\end{eqnarray}
and let $\mathcal
{E}_j^k=\xi^k e^{ij\theta}$,
$\theta\in[-\pi,\pi].$ Substituting it into the error equation gives
$$
\begin{array}{lll}
\displaystyle
\xi(\theta)=\frac{\left(s_1-s_2-s_3\nu\right)-is_4\nu\sin(\theta)}
{\left(s_1+s_2+s_3\nu\right)+is_4\nu\sin(\theta)}
\end{array}
$$
where
$$
\begin{array}{lll}
\displaystyle
s_1=\frac{2}{\tau}\left(1-\frac{1}{3}\sin^2\left(\frac{\theta}{2}\right)\right),\;\;
s_2=2\sin^2\left(\frac{\theta}{2}\right)\left(\frac{2d_2}{h^2}+\frac{d_1^2}{6d_2}\right),\vspace{0.3
cm}\\ \displaystyle
s_3=2\left(1-\frac{1}{3}\sin^2\left(\frac{\theta}{2}\right)\right)\sum
\limits_{\ell=0}^{\infty}\varpi_{4,\ell}^{(\alpha)}\cos(\ell\theta),s_4=\left(\frac{d_1h}{6\tau
d_2}+\frac{d_1}{h}\right)-\frac{d_1h}
{6d_2}\sum\limits_{\ell=0}^{\infty}\varpi_{4,\ell}^{(\alpha)}\cos(\ell\theta).
\end{array}
$$

Note that $\nu, s_2,s_3\geq0$. It follows from Theorem \ref{th.3} that
$%$
%\begin{array}{lll}
\displaystyle \left|\xi(\theta)\right|\leq1
%\end{array}
%$
$
if $\displaystyle 0<\alpha\leq 0.8439$. So
scheme (\ref{eq.15}) is unconditionally stable if $\displaystyle 0<\alpha\leq
0.8439$. For such $\alpha\in(0,0.8439]$, the convergence order of scheme (\ref{eq.15}) is
$\mathcal {O}(\tau^2+h^4)$.
\end{proof}

\subsection{The 6th-order scheme in space}

From Taylor expansion, we have
\begin{equation}
\displaystyle \frac{\partial u(x_j,t)}{\partial
x}=\mu_x\delta_{{x}}\left(I-\frac{h^2}{6}\delta_x^2\right)u(x_j,t)+\frac{h^4}{30}\frac{\partial^5
u(x_j,t)}{\partial x^5}+\mathcal {O}(h^6),\label{eq.17}
\end{equation}
\begin{equation}
\displaystyle \frac{\partial^2 u(x_j,t)}{\partial
x^2}=\delta_{{x}}^2\left(I-\frac{h^2}{12}\delta_x^2\right)u(x_j,t)+\frac{h^4}{90}\frac{\partial^6
u(x_j,t)}{\partial x^6}+\mathcal {O}(h^6),\label{eq.18}
\end{equation}
\begin{equation}
\displaystyle \frac{\partial^3 u(x_j,t)}{\partial
x^3}=\mu_x\delta_{{x}}^3u(x_j,t)+\mathcal {O}(h^2),\label{eq.19}
\end{equation}
\begin{equation}
\displaystyle \frac{\partial^4 u(x_j,t)}{\partial
x^4}=\delta_{{x}}^4u(x_j,t)+\mathcal {O}(h^2).\label{eq.20}
\end{equation}

According to (\ref{eq.12}), one gets
\begin{equation}
\displaystyle \frac{\partial^5 u(x_j,t)}{\partial
x^5}=\frac{d_1}{d_2}\frac{\partial^4 u(x_j,t)}{\partial
x^4}+\frac{1}{d_2}\frac{\partial^3 g(x_j,t)}{\partial x^3}\label{eq.21}
\end{equation}
and
\begin{eqnarray}
\begin{array}{lll}
\displaystyle \frac{\partial^6 u(x_j,t)}{\partial
x^6}=\frac{d_1^2}{d_2^2}\frac{\partial^4
u(x_j,t)}{\partial x^4}+\frac{1}{d_2}\frac{\partial^4
g(x_j,t)}{\partial x^4}+\frac{d_1}{d_2^2}\frac{\partial^3
g(x_j,t)}{\partial x^3}.\label{eq.22}
\end{array}
\end{eqnarray}
From equations (\ref{eq.17})-(\ref{eq.22}), one has
\begin{eqnarray}
\begin{array}{lll}
\displaystyle \frac{\partial u(x_j,t)}{\partial
x}=&\displaystyle\mu_x\delta_{{x}}\left(I-\frac{h^2}{6}\delta_x^2\right)u(x_j,t)
+\frac{d_1h^4}{30d_2}\delta_x^4u(x_j,t)\vspace{0.1 cm}\\
&\displaystyle
+\frac{h^4}{30d_2}\mu_x\delta_x^3g(x_j,t) +\mathcal {O}(h^6),\label{eq.23}
\end{array}
\end{eqnarray}
and
\begin{eqnarray}
\begin{array}{lll}
\displaystyle \frac{\partial^2 u(x_j,t)}{\partial
x^2}=&\displaystyle\delta_{{x}}^2\left(I-\frac{h^2}{12}\delta_x^2\right)u(x_j,t)
+\frac{d_1^2h^4}{90d_2^2}\delta_x^4u(x_j,t)\vspace{0.1 cm}\\
&\displaystyle+\frac{h^4}{90d_2}\delta_x^4g(x_j,t)
+\frac{h^4d_1}{90d_2^2}\mu_x\delta_x^3g(x_j,t) +\mathcal {O}(h^6).\label{eq.24}
\end{array}
\end{eqnarray}

Substituting (\ref{eq.23}) and (\ref{eq.24}) into (\ref{eq.12}) gives
\begin{eqnarray}
\begin{array}{lll}
\Delta_6u(x_j,t)=\widetilde{\Delta}_6g(x_j,t)+\mathcal {O}(h^6),\label{eq.25}
\end{array}
\end{eqnarray}
where
\begin{eqnarray*}
\begin{array}{lll}\displaystyle
\Delta_6=d_2\delta_x^2\left(I-\frac{h^2}{12}\delta_x^2\right)-d_1\mu_x\delta_{{x}}
\left(I-\frac{h^2}{6}\delta_x^2\right)-\frac{d_1^2h^4}{45d_2}\delta_x^4,\vspace{0.1 cm}\\\displaystyle
 \widetilde{\Delta}_6=I-\frac{h^4}{90}\delta_x^2\left(\delta_x^2-\frac{2d_1}{d_2}\mu_x\delta_{{x}}\right).
\end{array}
\end{eqnarray*}
Combing (\ref{eq.6}) with (\ref{eq.12}) and (\ref{eq.25}), we can obtain
\begin{eqnarray}
\begin{array}{lll}
&\displaystyle \frac{ 2}{\tau}\widetilde{\Delta}_6\cdot\left(u(x_j,t_{k+1})-u(x_j,t_{k})\right)=
 \Delta_6\cdot\left(u(x_j,t_{k+1})+u(x_j,t_{k})\right) \vspace{0.2 cm}\\
&\displaystyle +d_{\gamma}\widetilde{\Delta}_6\cdot\left(\frac{\partial
u(x_j,t_{k+1})}{\partial |x|^{\gamma}}+\frac{\partial
u(x_j,t_{k})}{\partial |x|^{\gamma}}\right)+2\widetilde{\Delta}_6\cdot
s(x_j,t_{k+\frac{1}{2}})\\&+\vspace{0.2 cm} \mathcal {O}(\tau^2+h^6).\label{eq.26}
\end{array}
\end{eqnarray}

For the Riesz derivative, we apply the sixth-order
numerical scheme
$$
\begin{array}{rrr}
\displaystyle \frac{\partial^\alpha u(x,t)}{\partial{|x|^\alpha}}=
-\frac{1}{2\cos\left(\pi\alpha/2\right)h^\alpha}\left
(\sum\limits_{\ell=0}^{\infty}\varpi_{6,\ell}^{(\alpha)}f(x-\ell
h)+\sum\limits_{\ell=0}^{\infty}\varpi_{6,\ell}^{(\alpha)}f(x+\ell
h)\right)+\mathcal {O}(h^6).
\end{array}
$$
 Substitution in (\ref{eq.26}) gives
\begin{eqnarray}
\begin{array}{lll}
\displaystyle \left(\frac{2c_1}{\tau}+e_1\right)u_{j-2}^{k+1} +
\left(\frac{2c_2}{\tau}+e_2\right)u_{j-1}^{k+1}+\left(\frac{2c_3}{\tau}+e_3\right)u_{j}^{k+1}\vspace{0.2 cm}\\ \displaystyle+
\left(\frac{2c_4}{\tau}+e_4\right)u_{j+1}^{k+1}
\displaystyle +\left(\frac{2c_5}{\tau}+e_5\right)u_{j+2}^{k+1}
+\nu\left[\sum\limits_{\ell=0}^{j}\varpi_{6,\ell}^{(\alpha)}\right.\vspace{0.2 cm}\\\left.\displaystyle
\left(c_1u_{j-l-2}^{k+1}+
c_2u_{j-l-1}^{k+1}+c_3u_{j-l}^{k+1}+c_4u_{j-l+1}^{k+1}+c_5u_{j-l+2}^{k+1}
\right)\right.\vspace{0.2 cm}\\ \displaystyle\left.
+\sum\limits_{\ell=0}^{M-j}\varpi_{6,\ell}^{(\alpha)}
\left(c_1u_{j+l-2}^{k+1}+
c_2u_{j+l-1}^{k+1}+c_3u_{j+l}^{k+1}+c_4u_{j+l+1}^{k+1}\right.\right.\vspace{0.2 cm}\\ \displaystyle\left.\left.+c_5u_{j+l+2}^{k+1}
\right) \right] =
\left(\frac{2c_1}{\tau}-e_1\right)u_{j-2}^{k} +
\left(\frac{2c_2}{\tau}-e_2\right)u_{j-1}^{k}\vspace{0.2 cm}\\ \displaystyle+\left(\frac{2c_3}{\tau}-e_3\right)u_{j}^{k}
\displaystyle+
\left(\frac{2c_4}{\tau}-e_4\right)u_{j+1}^{k} +\left(\frac{2c_5}{\tau}-e_5\right)u_{j+2}^{k}\vspace{0.2 cm}\\\displaystyle
-\nu\left[\sum\limits_{\ell=0}^{j}\varpi_{6,\ell}^{(\alpha)}
\left(c_1u_{j-l-2}^{k}+
c_2u_{j-l-1}^{k}+c_3u_{j-l}^{k+1}+c_4u_{j-l+1}^{k}\right.\right.\vspace{0.1 cm}\\\displaystyle\left.\left.\displaystyle+c_5u_{j-l+2}^{k} \right)
+\sum\limits_{\ell=0}^{M-j}\varpi_{6,\ell}^{(\alpha)}
\left(c_1u_{j+l-2}^{k}+
c_2u_{j+l-1}^{k}+c_3u_{j+l}^{k}\right.\right.\vspace{0.2 cm}\\ \displaystyle\left.\left.+c_4u_{j+l+1}^{k}+c_5u_{j+l+2}^{k}
\right) \right]+2
c_1f_{j-2}^{k+\frac{1}{2}}+
2c_2f_{j-1}^{k+\frac{1}{2}}\vspace{0.2 cm}\\+2c_3f_{j}^{k+\frac{1}{2}}+2c_4f_{j+1}^{k+\frac{1}{2}} \displaystyle
+2c_5f_{j+2}^{k+\frac{1}{2}},\;\;\; j=2,\ldots,M-2,\vspace{0.2 cm}\\ k=0,1,\ldots,N-1.\label{eq.27}
\end{array}
\end{eqnarray}
Here
$$
\begin{array}{lll}
\displaystyle
\nu=\frac{d_{\alpha}}{2\cos\left(\pi\alpha/2\right)h^\alpha},\;\;
c_1=-\frac{1}{90}\left(1+\frac{d_1h}{d_2}\right),\;\;
c_2=\frac{1}{90}\left(4+\frac{2d_1h}{d_2}\right),\;\;\vspace{0.2 cm}\\
\displaystyle
c_3=\frac{14}{15},\;\;c_4=\frac{1}{90}\left(4-\frac{2d_1h}{d_2}\right),\;\;
c_5=-\frac{1}{90}\left(1-\frac{d_1h}{d_2}\right),\vspace{0.2 cm}\\
\displaystyle
e_1=\frac{d_2}{12h^2}+\frac{d_1}{12h}+\frac{d_1^2}{45d_2},\;\;
e_2=-\left(\frac{4d_2}{3h^2}+\frac{2d_1}{3h}+\frac{4d_1^2}{45d_2}\right),\vspace{0.2 cm}\\
\displaystyle e_3=\frac{5d_2}{2h^2}+\frac{2d_1^2}{15d_2},\;
e_4=-\left(\frac{4d_2}{3h^2}-\frac{2d_1}{3h}+\frac{4d_1^2}{45d_2}\right),\;
e_5=\frac{d_2}{12h^2}-\frac{d_1}{12h}+\frac{d_1^2}{45d_2}.
\end{array}
$$

%Next, we use the same procedure to show that scheme (\ref{eq.27}) is unconditionally stable.

\begin{theorem} The numerical scheme (\ref{eq.27}) is
unconditionally stable and convergent with order $\mathcal O(\tau^2+h^6)$.
\end{theorem}

\begin{proof}
 Let $\mathcal
{E}_j^k=\xi^k e^{ij\theta}$,
$\theta\in[-\pi,\pi].$ Substituting it into the
error equation of the equation (\ref{eq.27})
$$
\begin{array}{lll}
\displaystyle \left(\frac{2c_1}{\tau}+e_1\right)\mathcal
{E}_{j-2}^{k+1} + \left(\frac{2c_2}{\tau}+e_2\right)\mathcal
{E}_{j-1}^{k+1}+\left(\frac{2c_3}{\tau}+e_3\right)\mathcal
{E}_{j}^{k+1}+
\left(\frac{2c_4}{\tau}+e_4\right)\mathcal {E}_{j+1}^{k+1}\\
\displaystyle +\left(\frac{2c_5}{\tau}+e_5\right)\mathcal
{E}_{j+2}^{k+1}
+\nu\left[\sum\limits_{\ell=0}^{j}\varpi_{6,\ell}^{(\alpha)}
\left(c_1\mathcal {E}_{j-l-2}^{k+1}+ c_2\mathcal
{E}_{j-l-1}^{k+1}+c_3\mathcal {E}_{j-l}^{k+1}+c_4\mathcal
{E}_{j-l+1}^{k+1}\right.\right.\\\left.\left.\displaystyle+c_5\mathcal {E}_{j-l+2}^{k+1}
\right)\displaystyle
+\sum\limits_{\ell=0}^{M-j}\varpi_{6,\ell}^{(\alpha)}
\left(c_1\mathcal {E}_{j+l-2}^{k+1}+ c_2\mathcal
{E}_{j+l-1}^{k+1}+c_3\mathcal {E}_{j+l}^{k+1}+c_4\mathcal
{E}_{j+l+1}^{k+1}+c_5\mathcal {E}_{j+l+2}^{k+1} \right)
\right]\\\displaystyle =
\left(\frac{2c_1}{\tau}-e_1\right)\mathcal {E}_{j-2}^{k} +
\left(\frac{2c_2}{\tau}-e_2\right)\mathcal
{E}_{j-1}^{k}+\left(\frac{2c_3}{\tau}-e_3\right)\mathcal
{E}_{j}^{k}+
\left(\frac{2c_4}{\tau}-e_4\right)\mathcal {E}_{j+1}^{k}\\
\displaystyle +\left(\frac{2c_5}{\tau}-e_5\right)\mathcal
{E}_{j+2}^{k}
-\nu\left[\sum\limits_{\ell=0}^{j}\varpi_{6,\ell}^{(\alpha)}
\left(c_1\mathcal {E}_{j-l-2}^{k}+ c_2\mathcal
{E}_{j-l-1}^{k}+c_3\mathcal {E}_{j-l}^{k+1}+c_4\mathcal
{E}_{j-l+1}^{k}\right.\right.\\\left.\left.\displaystyle+c_5\mathcal {E}_{j-l+2}^{k}
\right)\displaystyle
+\sum\limits_{\ell=0}^{M-j}\varpi_{6,\ell}^{(\alpha)}
\left(c_1\mathcal {E}_{j+l-2}^{k}+ c_2\mathcal
{E}_{j+l-1}^{k}+c_3\mathcal {E}_{j+l}^{k}+c_4\mathcal
{E}_{j+l+1}^{k}+c_5\mathcal {E}_{j+l+2}^{k} \right) \right],\\
\displaystyle\;\;\; j=2,\ldots,M-2,\;k=0,1,\ldots,N-1,
\end{array}
$$
leads to
$$
\begin{array}{lll}
\displaystyle
\xi(\theta)=\frac{\left(w_1-w_2-w_3\nu\right)-iw_4\sin(\theta)}
{\left(w_1+w_2+w_3\nu\right)+iw_4\sin(\theta)}.
\end{array}
$$
Here
$$
\begin{array}{lll}
\displaystyle
w_1=\frac{2}{45\tau}\left[45-8\sin^4\left(\frac{\theta}{2}\right)\right],\;\;
w_2=\frac{2d_2}{3h^2}(7-\cos\theta)\sin^2\left(\frac{\theta}{2}\right)+\frac{16d_1^2}{45d_2}\sin^4\left(\frac{\theta}{2}\right)
,\vspace{0.1 cm}\\ \displaystyle
w_3=\frac{2}{45}\left[45-8\sin^4\left(\frac{\theta}{2}\right)\right]\sum
\limits_{\ell=0}^{\infty}\varpi_{6,\ell}^{(\alpha)}\cos(\ell\theta),
\vspace{0.1 cm}\\ \displaystyle w_4=-\frac{8d_1h}{45\tau
d_2}\sin^2\left(\frac{\theta}{2}\right)+\frac{d_1}{3h}(4-\cos\theta)
+\frac{8d_1h\nu}{45d_2}\sin^2\left(\frac{\theta}{2}\right)
\sum\limits_{\ell=0}^{\infty}\varpi_{6,\ell}^{(\alpha)}\cos(\ell\theta).
\end{array}
$$

It is clear that $\nu, w_2,w_3\geq0$. By using Theorem \ref{th.2}, one also has
$$
\begin{array}{lll}
\displaystyle \left|\xi(\theta)\right|\leq1.
\end{array}
$$
So scheme (\ref{eq.27}) is unconditionally stable. It can be shown that the
convergence order of scheme (\ref{eq.27}) for equation (\ref{eq.1}) is $\mathcal {O}(\tau^2+h^6)$.
\end{proof}

In the next section, we present several numerical examples.

\section{Numerical examples}
Firstly, we test the higher-order schemes for Riesz derivatives.

 \textbf{Example 1.} Consider
the function
$f_p(x)=x^p(1-x)^p$, $x\in[0,1]$, $p=2,3,4,5,6.$

 The Riesz derivative of the above function is analytically expressed as
$$
\begin{array}{lll}
\displaystyle \frac{\partial^\alpha
f_p(x)}{\partial{|x|^\alpha}}=-\frac{1}{2\cos\left(\pi\alpha/2\right)}
\sum
\limits_{\ell=0}^{p}(-1)^{\ell}\frac{p!(p+l)!}{\ell!(p-l)!\Gamma(p+\ell+1-\alpha)}
\left[x^{p+\ell-\alpha}+(1-x)^{p+\ell-\alpha}\right].
\end{array}
$$

We numerically solve $f_p(x)$ by using numerical scheme (\ref{eq.5}).
 The numerical results are presented in Tables \ref{tab.1}-\ref{tab.5}. From these tables, the experimental
 orders are in line with the theoretical orders $p~(p=2,3,4,5,6)$.

\begin{table}[h]
\begin{center}
\caption{ The absolute error, convergence order of Example 1 by
numerical scheme (\ref{eq.5}) with $p=2$.}\label{tab.1}
 %\vspace{0.3 cm}
 \begin{footnotesize}
\begin{tabular}{c c c c c c }\hline
  $\alpha$ & $h$&  \textrm{the absolute error}&\textrm {the convergence order}
  \\ \hline
 % $0.1 $& $\frac{1}{20}$ & 9.799438e-005    &  ---\vspace{0.1 cm}\\
%  $$  & $\frac{1}{40}$&  2.468338e-005 &        1.9892\vspace{0.1 cm}\\
%  $$& $\frac{1}{80}$ & 6.158789e-006  &            2.0028\vspace{0.1 cm}\\
%  $$& $\frac{1}{160}$ & 1.536039e-006  &       2.0034\vspace{0.1 cm}\\
%   $$& $\frac{1}{320}$ &  3.834192e-007  &    2.0022\vspace{0.1 cm}\\ \hline
   $0.2 $& $\frac{1}{20}$ &2.381267e-004 & ---\vspace{0.1 cm}\\
  $$  & $\frac{1}{40}$&  5.900964e-005  &  2.0127\vspace{0.1 cm}\\
   $$  & $\frac{1}{80}$&1.460639e-005  &  2.0144\vspace{0.1 cm}\\
  $$& $\frac{1}{160}$ &3.628491e-006   &   2.0092\vspace{0.1 cm}\\
   $$& $\frac{1}{320}$ &  9.039358e-007  &   2.0051\vspace{0.1 cm}\\\hline
  %$0.3 $& $\frac{1}{20}$ &4.350970e-004  &     ---\vspace{0.1 cm}\\
%  $$  & $\frac{1}{40}$&  1.059862e-004  &  2.0375 \vspace{0.1 cm}\\
%  $$  & $\frac{1}{80}$& 2.601041e-005   &      2.0267\vspace{0.1 cm}\\
%  $$& $\frac{1}{160}$ &  6.433677e-006     &   2.0154  \vspace{0.1 cm}\\
%   $$& $\frac{1}{320}$ &  1.599305e-006    &  2.0082\vspace{0.1 cm}\\ \hline
   $0.4 $& $\frac{1}{20}$ &7.097814e-004  &      ---\vspace{0.1 cm}\\
  $$  & $\frac{1}{40}$& 1.696639e-004 &    2.0647\vspace{0.1 cm}\\
\hline
\end{tabular}
 \end{footnotesize}
 \end{center}
 \end{table}

\begin{table}[h]
\begin{center}
\caption{ The absolute error, convergence order of Example 1 by
numerical scheme (\ref{eq.5}) with $p=3$.}\label{tab.2}
% \vspace{0.3 cm}
 \begin{footnotesize}
\begin{tabular}{c c c c c c }\hline
  $\alpha$ &$h$&  \textrm{the absolute error}&$\textrm {the convergence order}$
  \\\hline
   $0.2 $& $\frac{1}{40}$ &3.146678e-007& ---\vspace{0.1 cm}\\
  $$  & $\frac{1}{60}$&  1.576085e-007  & 1.7052\vspace{0.1 cm}\\
   $$  & $\frac{1}{80}$&7.991483e-008  &  2.3608\vspace{0.1 cm}\\
  $$& $\frac{1}{100}$ &4.501080e-008   &   2.5726\vspace{0.1 cm}\\
   $$& $\frac{1}{120}$ & 2.761993e-008    &  2.6786\vspace{0.1 cm}\\\hline
% $0.3 $& $\frac{1}{40}$ &1.542398e-006  &     ---\vspace{0.1cm}\\
%  $$  & $\frac{1}{60}$& 5.781085e-007  &  2.4203\vspace{0.1cm}\\
%  $$  & $\frac{1}{80}$& 2.689075e-007  &     2.6606\vspace{0.1cm}\\
%  $$& $\frac{1}{100}$ & 1.452726e-007     &  2.7595   \vspace{0.1cm}\\
%   $$& $\frac{1}{120}$ &  8.697527e-008   &  2.8136\vspace{0.1 cm}\\ \hline
   $0.4 $& $\frac{1}{40}$ &4.349687e-006  &      ---\vspace{0.1cm}\\
  $$  & $\frac{1}{60}$& 1.491902e-006 &    2.6391\vspace{0.1cm}\\
\hline
\end{tabular}
 \end{footnotesize}
 \end{center}
 \end{table}

\begin{table}[h]
\begin{center}
\caption{ The absolute error, convergence order of Example 1 by
numerical scheme (\ref{eq.5}) with $p=4$.}\label{tab.3}
% \vspace{0.3 cm}
 \begin{footnotesize}
\begin{tabular}{c c c c c c }\hline
  $\alpha$ & $h$&  \textrm{the absolute error}&$\textrm {the convergence order}$
  \\\hline
   $0.2 $& $\frac{1}{20}$ &3.254967e-006& ---\vspace{0.1 cm}\\
  $$  & $\frac{1}{25}$& 1.421712e-006  & 3.7121\vspace{0.1 cm}\\
   $$  & $\frac{1}{30}$&7.061638e-007    & 3.8381\vspace{0.1 cm}\\
  $$& $\frac{1}{35}$ &3.863551e-007  &       3.9123\vspace{0.1 cm}\\
   $$& $\frac{1}{40}$ &2.279637e-007   & 3.9509\vspace{0.1 cm}\\\hline
  %$0.3 $& $\frac{1}{20}$ &    6.905005e-006 &     ---\vspace{0.1cm}\\
%  $$  & $\frac{1}{25}$& 2.939291e-006  &  3.8275\vspace{0.1cm}\\
%  $$  & $\frac{1}{30}$& 1.435395e-006  &    3.9311\vspace{0.1cm}\\
%  $$& $\frac{1}{35}$ &7.762396e-007    &  3.9879\vspace{0.1cm}\\
%   $$& $\frac{1}{40}$ & 4.540565e-007  &  4.0158\\ \hline
   $0.4 $& $\frac{1}{20}$ &1.307893e-005  &      ---\vspace{0.1cm}\\
  $$  & $\frac{1}{25}$& 5.433994e-006  &   3.9362\vspace{0.1cm}\\
\hline
\end{tabular}
 \end{footnotesize}
 \end{center}
 \end{table}

\begin{table}[h]
\begin{center}
\caption{ The absolute error, convergence order of Example 1 by
numerical scheme (\ref{eq.5}) with $p=5$.}\label{tab.4}
 %\vspace{0.1 cm}
 \begin{footnotesize}
\begin{tabular}{c c c c c c }\hline
  $\alpha$ & $h$&  \textrm{the absolute error}&$\textrm {the convergence order}$
  \\\hline
  $0.2 $& $\frac{1}{80}$ &3.254967e-006& ---\vspace{0.1cm}\\
  $$  & $\frac{1}{100}$& 1.421712e-006 & 3.7121\vspace{0.1cm}\\
   $$  & $\frac{1}{120}$&7.061638e-007    & 3.8381\vspace{0.1cm}\\
  $$& $\frac{1}{140}$ &3.863551e-007  &       3.9123\vspace{0.1cm}\\
   $$& $\frac{1}{160}$ &2.279637e-007     & 3.9509\vspace{0.1 cm}\\\hline
  %$0.3 $& $\frac{1}{80}$ &   2.498729e-010  &     ---\vspace{0.1 cm}\\
%  $$  & $\frac{1}{100}$& 1.091359e-010& 3.7122\vspace{0.1 cm}\\
%  $$  & $\frac{1}{120}$& 5.108016e-011  &   4.1641\vspace{0.1 cm}\\
%  $$& $\frac{1}{140}$ &2.599689e-011    & 4.3816\vspace{0.1 cm}\\
%   $$& $\frac{1}{160}$ &1.423638e-011     & 4.5096\vspace{0.1 cm}\\ \hline
   $0.4 $& $\frac{1}{80}$ & 8.731739e-010    &      ---\vspace{0.1cm}\\
  $$  & $\frac{1}{100}$&3.348011e-010   &   4.2959\vspace{0.1cm}\\
\hline
\end{tabular}
 \end{footnotesize}
 \end{center}
 \end{table}

\begin{table}[h]
\begin{center}
\caption{ The absolute error, convergence order of Example 1 by
numerical scheme (\ref{eq.5}) with $p=6$.}\label{tab.5}
% \vspace{0.1 cm}
 \begin{footnotesize}
\begin{tabular}{c c c c c c }\hline
  $\alpha$ & $h$&  \textrm{the absolute error}&$\textrm {the convergence order}$
  \\\hline
   $0.2 $& $\frac{1}{20}$ &3.783855e-008& ---\vspace{0.1cm}\\
  $$  & $\frac{1}{40}$& 2.310553e-009  &    4.0335\vspace{0.1cm}\\
   $$  & $\frac{1}{80}$&4.258651e-011    &     5.7617\vspace{0.1cm}\\
  $$& $\frac{1}{160}$ &6.690552e-013  &      5.9921\vspace{0.1cm}\\
   $$& $\frac{1}{320}$ &1.035841e-014   & 6.0133\vspace{0.1 cm}\\\hline
  %$0.3 $& $\frac{1}{20}$ & 1.224183e-007  &     ---\vspace{0.1 cm}\\
%  $$  & $\frac{1}{40}$& 5.170505e-009  &     4.5654\vspace{0.1 cm}\\
%  $$  & $\frac{1}{80}$& 8.866763e-011  &    5.8658\vspace{0.1 cm}\\
%  $$& $\frac{1}{160}$ &1.345871e-012   &   6.0418\vspace{0.1 cm}\\
%   $$& $\frac{1}{320}$ &2.028938e-014        &     6.0517\vspace{0.1 cm}\\ \hline
   $0.4 $& $\frac{1}{20}$ &3.116503e-007   &      ---\vspace{0.1cm}\\
  $$  & $\frac{1}{40}$&    1.031745e-008   &   4.9168\vspace{0.1cm}\\
\hline
\end{tabular}
 \end{footnotesize}
 \end{center}
 \end{table}

Next we test the numerical schemes for the  equations which have the form of equ. (\ref{eq.1}).

\textbf{Example 2.} Consider the following equation
$$
\begin{array}{rrr}
\displaystyle \frac{\partial u(x,t)}{\partial t} =\frac{\partial^2
u(x,t)}{\partial x^2}-\frac{\partial u(x,t)}{\partial
x}+\frac{\partial^{\alpha} u(x,t)}{\partial
|x|^{\alpha}}+s(x,t),\;\; 0\leq x\leq 1,\;0\leq t\leq1,
\end{array}
$$
%subject to the initial and boundary values conditions
%$$
%\begin{array}{rrr}
%\displaystyle u(x,0)=x^6(1-x)^6, \;\;0\leq x\leq 1,
%\end{array}
%$$
%and
%$$
%\begin{array}{rrr}
%\displaystyle u(0,t)=0, \;\;  u(1,t)=0,\;\;0\leq t\leq 1,
%\end{array}
%$$
 in which
$$
\begin{array}{lll}
\displaystyle f(x,t)&=&\displaystyle\exp(t)x^4(1-x)^4(x^4+10x^3-149
x^2+138x-30)+\frac{1}{2}\exp(t)\sec\left(\frac{\pi}{2}\alpha\right)\vspace{0.2 cm}\\&& \displaystyle\left\{\frac{\Gamma(7)}{\Gamma(7-\alpha)}
\left[x^{6-\alpha}+(1-x)^{6-\alpha}\right]
-\frac{6\Gamma(8)}{\Gamma(8-\alpha)}
\left[x^{7-\alpha}+(1-x)^{7-\alpha}\right]\right.\vspace{0.2 cm}\\&&\displaystyle +\frac{15\Gamma(9)}{\Gamma(9-\alpha)}
\left[x^{8-\alpha}+(1-x)^{8-\alpha}\right]-\frac{20\Gamma(10)}{\Gamma(10-\alpha)}
\left[x^{9-\alpha}+(1-x)^{9-\alpha}\right]\vspace{0.2 cm}\\&&
\displaystyle +\frac{15\Gamma(11)}{\Gamma(11-\alpha)}
\left[x^{10-\alpha}+(1-x)^{10-\alpha}\right]
-\frac{6\Gamma(12)}{\Gamma(12-\alpha)}
\left[x^{11-\alpha}+(1-x)^{11-\alpha}\right]\vspace{0.2 cm}\\&&
\displaystyle\left. +\frac{\Gamma(13)}{\Gamma(13-\alpha)}
\left[x^{12-\alpha}+(1-x)^{12-\alpha}\right] \right\}.
\end{array}
$$
Its analytical solution is $u(x,t)=\exp(t)x^6(1-x)^6$ {and satisfy the corresponding initial and boundary values conditions.}

We solve this problem with the numerical schemes (\ref{eq.10}) and (\ref{eq.15}) for
different values of $\tau$, $h$ and $\alpha$. The
absolute error, temporal and spatial convergence orders are listed
in Tables \ref{tab.6} and \ref{tab.7}, which display that the numerical results are in line with our
theoretical analysis.

In the following, we give a slightly different example.

\begin{table}[h]
\begin{center}
\caption{ The absolute errors, temporal and spatial convergence
orders of Example 2 by difference scheme (\ref{eq.10}).}\label{tab.6}
% \vspace{0.3 cm}
 \begin{footnotesize}
\begin{tabular}{c c c c c c }\hline
   &   &&temporal &spatial
   \\
  $\alpha$ &&  the maximum errors & convergence
orders& convergence orders
  \\
  \hline
  $0.2 $& $h=\frac{1}{10},\tau=\frac{1}{10}$ &   2.581219e-005&  ---&    ---
\vspace{0.1 cm}\\
  $$  & $h=\frac{1}{20},\tau=\frac{1}{20}$&       6.217660e-006 & 2.0536&    2.0536\vspace{0.1 cm}\\
  $$& $h=\frac{1}{40},\tau=\frac{1}{40}$ &        1.536085e-006&    2.0171 &    2.0171 \vspace{0.1 cm}\\
  $$& $h=\frac{1}{80},\tau=\frac{1}{80}$ &      3.844480e-007&    1.9984 &     1.9984\vspace{0.1 cm}\\
 \hline
 $0.3 $& $h=\frac{1}{10},\tau=\frac{1}{10}$ &   2.514418e-005&  ---&    ---
\vspace{0.1 cm}\\
  $$  & $h=\frac{1}{20},\tau=\frac{1}{20}$&        6.061411e-006 &    2.0525&      2.0525\vspace{0.1 cm}\\
  $$& $h=\frac{1}{40},\tau=\frac{1}{40}$ &       1.498825e-006&   2.0158 &    2.0158 \vspace{0.1 cm}\\
  $$& $h=\frac{1}{80},\tau=\frac{1}{80}$ &       3.755193e-007&   1.9969 &    1.9969 \vspace{0.1 cm}\\
 \hline
 $0.4 $& $h=\frac{1}{10},\tau=\frac{1}{10}$ &    2.416676e-005&  ---&    ---
\vspace{0.1 cm}\\
  $$  & $h=\frac{1}{20},\tau=\frac{1}{20}$&        5.842791e-006&   2.0483&    2.0483\vspace{0.1 cm}\\
  $$& $h=\frac{1}{40},\tau=\frac{1}{40}$ &        1.448271e-006&    2.0123 &      2.0123 \vspace{0.1 cm}\\
  $$& $h=\frac{1}{80},\tau=\frac{1}{80}$ &       3.636282e-007&   1.9938 &     1.9938 \vspace{0.1 cm}\\
 \hline
  $0.5 $& $h=\frac{1}{10},\tau=\frac{1}{10}$ &    2.270557e-005&  ---&    ---
\vspace{0.1 cm}\\
  $$  & $h=\frac{1}{20},\tau=\frac{1}{20}$&        5.532646e-006 &   2.0370&    2.0370\vspace{0.1 cm}\\
  $$& $h=\frac{1}{40},\tau=\frac{1}{40}$ &       1.379247e-006&    2.0041 &      2.0041 \vspace{0.1 cm}\\
  $$& $h=\frac{1}{80},\tau=\frac{1}{80}$ &      3.477777e-007&   1.9876 &     1.9876\vspace{0.1 cm}\\
\hline
  $0.6 $& $h=\frac{1}{10},\tau=\frac{1}{10}$ &    2.043415e-005&  ---&    ---
\vspace{0.1 cm}\\
  $$  & $h=\frac{1}{20},\tau=\frac{1}{20}$&        5.079755e-006 &   2.0082&     2.0082\vspace{0.1 cm}\\
  $$& $h=\frac{1}{40},\tau=\frac{1}{40}$ &       1.283408e-006&    1.9848 &      1.9848 \vspace{0.1 cm}\\
  $$& $h=\frac{1}{80},\tau=\frac{1}{80}$ &     3.264917e-007&   1.9749 &     1.9749 \vspace{0.1 cm}\\
 \hline
  $0.7 $& $h=\frac{1}{10},\tau=\frac{1}{10}$ &    1.664449e-005&  ---&    ---
\vspace{0.1 cm}\\
  $$  & $h=\frac{1}{20},\tau=\frac{1}{20}$&         4.378341e-006 &   1.9266&    1.9266\vspace{0.1 cm}\\
  $$& $h=\frac{1}{40},\tau=\frac{1}{40}$ &        1.145019e-006&    1.9350 &     1.9350 \vspace{0.1 cm}\\
  $$& $h=\frac{1}{80},\tau=\frac{1}{80}$ &       2.972789e-007&    1.9455 &     1.9455 \vspace{0.1 cm}\\
 \hline
 %$0.8 $& $h=\frac{1}{10},\tau=\frac{1}{10}$ &   1.157882e-005&  ---&    ---
%\vspace{0.1 cm}\\
%  $$  & $h=\frac{1}{20},\tau=\frac{1}{20}$&       3.272218e-006 &   1.8231&     1.8231\vspace{0.1 cm}\\
%  $$& $h=\frac{1}{40},\tau=\frac{1}{40}$ &        9.340389e-007&   1.8087 &     1.8087\vspace{0.1 cm}\\
%  $$& $h=\frac{1}{80},\tau=\frac{1}{80}$ &      2.545979e-007&   1.8753 &    1.8753\vspace{0.1 cm}\\
% \hline
%  $0.9 $& $h=\frac{1}{10},\tau=\frac{1}{10}$ &   1.160708e-005&  ---&    ---
%\vspace{0.1 cm}\\
%  $$  & $h=\frac{1}{20},\tau=\frac{1}{20}$&        2.102472e-006&    2.4648&     2.4648\vspace{0.1 cm}\\
%  $$& $h=\frac{1}{40},\tau=\frac{1}{40}$ &        6.774367e-007&     1.6339 &      1.6339 \vspace{0.1 cm}\\
%  $$& $h=\frac{1}{80},\tau=\frac{1}{80}$ &       1.969763e-007&      1.7821 &     1.7821\vspace{0.1 cm}\\
% \hline
\end{tabular}
 \end{footnotesize}
 \end{center}
 \end{table}

\begin{table}[h]
\begin{center}
\caption{ The maximum errors, temporal and spatial convergence
orders of Example 2 by difference scheme (\ref{eq.15}).}\label{tab.7}
% \vspace{0.3 cm}
 \begin{footnotesize}
\begin{tabular}{c c c c c c }\hline
   &   &&temporal &spatial\\
  $\alpha$ &&  the maximum errors & convergence
orders& convergence orders
  \\\hline
  $0.2 $& $h=\frac{1}{4},\tau=\frac{1}{4}$ &   1.151043e-004&  ---&    ---
\vspace{0.1 cm}\\
  $$  & $h=\frac{1}{8},\tau=\frac{1}{16}$&       4.361384e-006&   2.3610&     4.7220\vspace{0.1 cm}\\
  $$& $h=\frac{1}{16},\tau=\frac{1}{64}$ &       2.346706e-007&   2.1081 &     4.2161\vspace{0.1 cm}\\
  $$& $h=\frac{1}{32},\tau=\frac{1}{256}$ &       2.036847e-008&    1.7631 &    3.5262 \vspace{0.1 cm}\\
 \hline
  $0.3 $& $h=\frac{1}{4},\tau=\frac{1}{4}$ &   1.128702e-004&  ---&    ---
\vspace{0.1 cm}\\
  $$  & $h=\frac{1}{8},\tau=\frac{1}{16}$&       4.362181e-006 &   2.3468&     4.6935\vspace{0.1 cm}\\
  $$& $h=\frac{1}{16},\tau=\frac{1}{64}$ &       2.461975e-007&    2.0736 &      4.1472\vspace{0.1 cm}\\
  $$& $h=\frac{1}{32},\tau=\frac{1}{256}$ &        2.396154e-008&   1.6805 &     3.3610\vspace{0.1 cm}\\
 \hline
  $0.4 $& $h=\frac{1}{4},\tau=\frac{1}{4}$ &    1.088961e-004&  ---&    ---
\vspace{0.1 cm}\\
  $$  & $h=\frac{1}{8},\tau=\frac{1}{16}$&        4.433621e-006 &   2.3091&    4.6183\vspace{0.1 cm}\\
  $$& $h=\frac{1}{16},\tau=\frac{1}{64}$ &        2.870612e-007&     1.9746 &    3.9491\vspace{0.1 cm}\\
  $$& $h=\frac{1}{32},\tau=\frac{1}{256}$ &       2.910816e-008&  1.6509 &     3.3019\vspace{0.1 cm}\\
 \hline
  $0.5 $& $h=\frac{1}{4},\tau=\frac{1}{4}$ &  1.018053e-004&  ---&    ---
\vspace{0.1 cm}\\
  $$  & $h=\frac{1}{8},\tau=\frac{1}{16}$&     4.654112e-006 &  2.2256&      4.4512\vspace{0.1 cm}\\
  $$& $h=\frac{1}{16},\tau=\frac{1}{64}$ &       3.607253e-007&    1.8447 &    3.6895 \vspace{0.1 cm}\\
  $$& $h=\frac{1}{32},\tau=\frac{1}{256}$ &      3.674011e-008&   1.6478 &     3.2955\vspace{0.1 cm}\\
 \hline
  $0.6 $& $h=\frac{1}{4},\tau=\frac{1}{4}$ &    8.897936e-005&  ---&    ---
\vspace{0.1 cm}\\
  $$  & $h=\frac{1}{8},\tau=\frac{1}{16}$&        5.201584e-006 &   2.0482&    4.0964\vspace{0.1 cm}\\
  $$& $h=\frac{1}{16},\tau=\frac{1}{64}$ &        4.980729e-007&   1.6923 &     3.3845\vspace{0.1 cm}\\
  $$& $h=\frac{1}{32},\tau=\frac{1}{256}$ &       4.862623e-008&    1.6783 &      3.3566\vspace{0.1 cm}\\
 \hline
  $0.7 $& $h=\frac{1}{4},\tau=\frac{1}{4}$ &    6.521192e-005&  ---&    ---
\vspace{0.1 cm}\\
  $$  & $h=\frac{1}{8},\tau=\frac{1}{16}$&       6.540139e-006 &   1.6588&     3.3177\vspace{0.1 cm}\\
  $$& $h=\frac{1}{16},\tau=\frac{1}{64}$ &        7.724068e-007&  1.5410 &    3.0819 \vspace{0.1 cm}\\
  $$& $h=\frac{1}{32},\tau=\frac{1}{256}$ &       6.867979e-008&   1.7457 &      3.4914 \vspace{0.1 cm}\\
 \hline
\end{tabular}
 \end{footnotesize}
 \end{center}
 \end{table}

\textbf{Example 3.} Consider the following equation
$$
\begin{array}{rrr}
\displaystyle \frac{\partial u(x,t)}{\partial t} =\frac{\partial^2
u(x,t)}{\partial x^2}-2\frac{\partial u(x,t)}{\partial
x}+\alpha^2\frac{\partial^{\alpha} u(x,t)}{\partial
|x|^{\alpha}}+s(x,t),\;\; 0\leq x\leq 1,\;0\leq t\leq1,
\end{array}
$$
%where initial and boundary values conditions are
%$$
%\begin{array}{rrr}
%\displaystyle u(x,0)=0, \;\;0\leq x\leq 1,
%\end{array}
%$$
%and
%$$
%\begin{array}{rrr}
%\displaystyle u(0,t)=0, \;\;  u(1,t)=0,\;\;0\leq t\leq 1,
%\end{array}
%$$
 in which
$$
\begin{array}{lll}
\displaystyle s(x,t)&=&\displaystyle
x^6(1-x)^6\left[\cos(t)(x^4-2x^3+ x^2)+\sin(t)(32x^3-288x^2+256x-56)
\right]\vspace{0.2 cm}\\&&
\displaystyle+\frac{\alpha^2}{2}\sin(t)\sec\left(\frac{\pi}{2}\alpha\right)\left\{\frac{\Gamma(9)}{\Gamma(9-\alpha)}
\left[x^{8-\alpha}+(1-x)^{8-\alpha}\right] \right.\vspace{0.2 cm}\\&&\displaystyle -\frac{8\Gamma(10)}{\Gamma(10-\alpha)}
\left[x^{9-\alpha}+(1-x)^{9-\alpha}\right]+\frac{28\Gamma(11)}{\Gamma(11-\alpha)}
\left[x^{10-\alpha}+(1-x)^{10-\alpha}\right]\vspace{0.2 cm}\\&&
\displaystyle -\frac{56\Gamma(12)}{\Gamma(12-\alpha)}
\left[x^{11-\alpha}+(1-x)^{11-\alpha}\right]
+\frac{70\Gamma(13)}{\Gamma(13-\alpha)}
\left[x^{12-\alpha}+(1-x)^{12-\alpha}\right]\vspace{0.2 cm}\\&&
\displaystyle -\frac{56\Gamma(14)}{\Gamma(14-\alpha)}
\left[x^{13-\alpha}+(1-x)^{13-\alpha}\right]
+\frac{28\Gamma(15)}{\Gamma(15-\alpha)}
\left[x^{14-\alpha}+(1-x)^{14-\alpha}\right] \vspace{0.2 cm}\\&&
\displaystyle\left. -\frac{8\Gamma(16)}{\Gamma(16-\alpha)}
\left[x^{15-\alpha}+(1-x)^{15-\alpha}\right]
+\frac{\Gamma(17)}{\Gamma(17-\alpha)}
\left[x^{16-\alpha}+(1-x)^{16-\alpha}\right] \right\}.
\end{array}
$$
Its exact solution is $u(x,t)=\sin(t)x^8(1-x)^8$ and satisfy the accordingly initial and boundary values conditions.

The absolute error, temporal and spatial convergence orders are
listed in Table \ref{tab.8} by numerical scheme (\ref{eq.27}). The numerical results agree with the theoretical results.

\begin{table}[h]
\begin{center}
\caption{ The maximum errors, temporal and spatial convergence
orders of Example 3 by difference scheme (\ref{eq.27}).}\label{tab.8}
% \vspace{0.3 cm}
 \begin{footnotesize}
\begin{tabular}{c c c c c c }\hline
   &   &&temporal &spatial
   \\
  $\alpha$ &&  the maximum errors& convergence
orders& convergence orders
  \\
  \hline
  $0.2 $& $h=\frac{1}{8},\tau=\frac{1}{8}$ &    1.360207e-007&  ---& ---\vspace{0.1 cm}\\
  $$  & $h=\frac{1}{16},\tau=\frac{1}{64}$&        2.071201e-009 &     2.0124&     6.0372\vspace{0.1 cm}\\
  $$& $h=\frac{1}{32},\tau=\frac{1}{512}$ &       3.348089e-011&    1.9837 &    5.9510 \vspace{0.1 cm}\\
  $$& $h=\frac{1}{64},\tau=\frac{1}{4096}$ &       5.235085e-013&   1.9997 &     5.9990 \vspace{0.1 cm}\\
 \hline
  $0.3 $& $h=\frac{1}{8},\tau=\frac{1}{8}$ &    1.356431e-007&  ---&    ---\vspace{0.1 cm}\\
  $$  & $h=\frac{1}{16},\tau=\frac{1}{64}$&        2.092867e-009 &   2.0061&    6.0182\vspace{0.1 cm}\\
  $$& $h=\frac{1}{32},\tau=\frac{1}{512}$ &        3.254863e-011&    2.0022 &      6.0067 \vspace{0.1 cm}\\
  $$& $h=\frac{1}{64},\tau=\frac{1}{4096}$ &        4.855887e-013&   2.0222 &     6.0667 \vspace{0.1 cm}\\
 \hline
  $0.4 $& $h=\frac{1}{8},\tau=\frac{1}{8}$ &    1.348600e-007&  ---&    ---\vspace{0.1 cm}\\
  $$  & $h=\frac{1}{16},\tau=\frac{1}{64}$&        2.146379e-009 &   1.9911&    5.9734\vspace{0.1 cm}\\
  $$& $h=\frac{1}{32},\tau=\frac{1}{512}$ &        2.972961e-011&   2.0580&     6.1739 \vspace{0.1 cm}\\
  $$& $h=\frac{1}{64},\tau=\frac{1}{4096}$ &        3.852828e-013&    2.0899 &      6.2698\vspace{0.1 cm}\\
 \hline
  $0.5 $& $h=\frac{1}{8},\tau=\frac{1}{8}$ &    1.335205e-007&  ---&    ---\vspace{0.1 cm}\\
  $$  & $h=\frac{1}{16},\tau=\frac{1}{64}$&       2.246228e-009&   1.9645&     5.8934\vspace{0.1 cm}\\
  $$& $h=\frac{1}{32},\tau=\frac{1}{512}$ &       2.200601e-011&    2.2245 &     6.6735 \vspace{0.1 cm}\\
  $$& $h=\frac{1}{64},\tau=\frac{1}{4096}$ &      2.816274e-013&  2.0960 &      6.2880\vspace{0.1 cm}\\
 \hline
  $0.6 $& $h=\frac{1}{8},\tau=\frac{1}{8}$ &    1.322258e-007&  ---&    ---\vspace{0.1 cm}\\
  $$  & $h=\frac{1}{16},\tau=\frac{1}{64}$&       2.372915e-009&   1.9334&      5.8002\vspace{0.1 cm}\\
  $$& $h=\frac{1}{32},\tau=\frac{1}{512}$ &      1.827817e-011&   2.3401&    7.0204\vspace{0.1 cm}\\
  $$& $h=\frac{1}{64},\tau=\frac{1}{4096}$ &       4.506292e-013&   1.7807&     5.3420 \vspace{0.1 cm}\\
 \hline
  $0.7 $& $h=\frac{1}{8},\tau=\frac{1}{8}$ &   1.357968e-007&  ---&    ---\vspace{0.1 cm}\\
  $$  & $h=\frac{1}{16},\tau=\frac{1}{64}$&        3.805230e-009&   1.7191&    5.1573\vspace{0.1 cm}\\
  $$& $h=\frac{1}{32},\tau=\frac{1}{512}$ &       6.671019e-011&   1.9446&     5.8339\vspace{0.1 cm}\\
  $$& $h=\frac{1}{64},\tau=\frac{1}{4096}$ &       1.819328e-012&  1.7321 &      5.1964 \vspace{0.1 cm}\\
 \hline
%   $0.8 $& $h=\frac{1}{8},\tau=\frac{1}{8}$ &   3.352575e-007&  ---&    ---\vspace{0.1 cm}\\
%  $$  & $h=\frac{1}{16},\tau=\frac{1}{64}$&        8.290612e-009 &   1.7792&     5.3376\vspace{0.1 cm}\\
%  $$& $h=\frac{1}{32},\tau=\frac{1}{512}$ &       2.525535e-010&     1.6789 &     5.0368 \vspace{0.1 cm}\\
%  $$& $h=\frac{1}{64},\tau=\frac{1}{4096}$ &       5.400819e-012&   1.8491 &      5.5473 \vspace{0.1 cm}\\
% \hline
%$0.9 $& $h=\frac{1}{8},\tau=\frac{1}{8}$ & 1.453414e-006&---&---\vspace{0.1 cm}\\
%  $$  & $h=\frac{1}{16},\tau=\frac{1}{64}$&        2.490304e-008 &     1.9557&     5.8670\vspace{0.1 cm}\\
%  $$& $h=\frac{1}{32},\tau=\frac{1}{512}$ &       9.940677e-010&   1.5489 &     4.6468 \vspace{0.1 cm}\\
%  $$& $h=\frac{1}{64},\tau=\frac{1}{4096}$ &       1.829541e-011&   1.9213 &     5.7638 \vspace{0.1 cm}\\
% \hline
\end{tabular}
 \end{footnotesize}
 \end{center}
 \end{table}

\section{Conclusions}
In this paper, we construct high-order (from 2nd-order to 6th-order)
numerical schemes to approximate the Riesz derivatives.
Next, we develop three kinds of difference schemes for the Riesz space fractional 
turbulent diffusion equation. The stability of the
derived numerical algorithms are shown by the Fourier method. The
convergence orders are $\mathcal {O}(\tau^2+h^2)$, $\mathcal
{O}(\tau^2+h^4)$ and $\mathcal {O}(\tau^2+h^6)$, respectively.
Finally, numerical results confirm the theoretical analysis.\\

\textbf{\emph{Appendix A}}

\begin{proof}
 (1) Direct calculations can finish it so we omit the details.

(2) See \cite{liding3} for details.

(3) Now we show the case $\alpha\in(1,2)$.
 We firstly show that $\varpi_{2,j}^{(\alpha)}>0$
for $\ell\geq5$. For convenience, denote $\alpha=1+\gamma$, where $0<\gamma<1$. Lengthy calculations give
$$
\begin{array}{lll}\label{5}
 \displaystyle\varpi_{2,\ell}^{(\alpha)}&=&\displaystyle
\left(\frac{3}{2}\right)^{1+\gamma}\sum\limits_{m=0}^{\ell}
\left(\frac{1}{3}\right)^{m}\varpi_{1,m}^{(1+\gamma)}
\varpi_{1,\ell-m}^{(1+\gamma)}
\vspace{0.2 cm}\\&=&
 \displaystyle\left(\frac{3}{2}\right)^{1+\gamma}\left[
\varpi_{1,0}^{(1+\gamma)} \varpi_{1,\ell}^{(1+\gamma)}
+\frac{1}{3}\varpi_{1,1}^{(1+\gamma)} \varpi_{1,\ell-1}^{(1+\gamma)}
+\frac{1}{9}\varpi_{1,2}^{(1+\alpha)} \varpi_{1,\ell-2}^{(1+\gamma)}
\right. \vspace{0.2 cm}\\&&\left.
 \displaystyle+
\left(\frac{1}{3}\right)^{\ell}\varpi_{1,0}^{(1+\gamma)}
\varpi_{1,\ell}^{(1+\gamma)}+
\left(\frac{1}{3}\right)^{\ell-1}\varpi_{1,1}^{(1+\gamma)}
\varpi_{1,\ell-1}^{(1+\gamma)}\right]
\vspace{0.2 cm}\\&&+ \displaystyle
 \left(\frac{3}{2}\right)^{1+\gamma}\sum\limits_{m=3}^{\ell-2}
\left(\frac{1}{3}\right)^{m}\varpi_{1,m}^{(1+\gamma)}
\varpi_{1,\ell-m}^{(1+\gamma)} \vspace{0.2 cm}\\&=&\displaystyle
 \left(\frac{3}{2}\right)^{1+\gamma}\left[
1-\frac{\gamma+1}{3}\frac{\ell}{(\ell-2-\gamma)}
+\frac{\gamma(\gamma+1)}{18}\frac{\ell(\ell-1)}{(\ell-2-\gamma)(\ell-3-\gamma)}
\right. \vspace{0.2 cm}\\&&\left.\displaystyle+
\left(\frac{1}{3}\right)^{\ell}\left(1-\frac{3\ell(\gamma+1)}{(\ell-2-\gamma)}\right)\right]\varpi_{1,\ell}^{(1+\gamma)}
+\displaystyle
 \left(\frac{3}{2}\right)^{1+\gamma}\sum\limits_{m=3}^{\ell-2}
\left(\frac{1}{3}\right)^{m}\varpi_{1,m}^{(1+\gamma)}
\varpi_{1,\ell-m}^{(1+\gamma)}
% \end{array}
%$$
%$$
%\begin{array}{lll}
 \vspace{0.2 cm}\\&\geq&\displaystyle\left(\frac{3}{2}\right)^{1-\alpha}\left[
\frac{244}{243} -\frac{28(\gamma+1)}{81}\frac{\ell}{(\ell-2-\gamma)}
+\frac{\gamma(\gamma+1)}{18}\frac{\ell(\ell-1)}{(\ell-2-\gamma)^2}
\right]\varpi_{1,\ell}^{(1+\gamma)} \vspace{0.2 cm}\\&&\displaystyle+
 \left(\frac{3}{2}\right)^{1+\gamma}\sum\limits_{m=3}^{\ell-2}
\left(\frac{1}{3}\right)^{m}\varpi_{1,m}^{(1+\gamma)}
\varpi_{1,\ell-m}^{(1+\gamma)},\;\;\;\ell\geq5.
\end{array}
$$

Let
$$
\begin{array}{lll}
\displaystyle F(x,\gamma)=\frac{244}{243}
-\frac{28(\gamma+1)}{81}\frac{x}{(x-2-\gamma)}
+\frac{\gamma(\gamma+1)}{18}\frac{x(x-1)}{(x-2-\gamma)^2},\;x\geq5,
\end{array}
$$
and
$$
\begin{array}{lll}
\displaystyle
G(x,\gamma)\hspace{-0.2cm}&=&\hspace{-0.2cm}\displaystyle486(x-2-\gamma)^2F(x,\gamma)\hspace{0.3cm}
\vspace{0.2 cm}\\&=&\hspace{-0.2cm}\displaystyle
488(x-2-\gamma)^2-168(\gamma+1)x(x-2-\gamma)+27\gamma(\gamma+1)x(x-1),\;x\geq5.
\end{array}
$$
Then
$$
\begin{array}{lll}
\displaystyle G_x(x,\gamma)=
976(x-2-\gamma)-168(\gamma+1)(2x-2-\gamma)+27\gamma(\gamma+1)(2x-1)
\end{array}
$$
and
$$
\begin{array}{lll}
\displaystyle G_{xx}(x,\gamma)= 54\gamma^2-282\gamma+640.
\end{array}
$$
Obviously,
$$
\begin{array}{lll}
\displaystyle G_{xx}(x,\gamma)\geq0\;\; \textmd{for} \;\;0<\gamma<1,
\end{array}
$$
it immediately follows that $G_x(x,\gamma)$ is an increasing function and
$ G_x(x,\gamma)\geq G_x(5,\gamma)$ for $x\geq5$.

Note that
$$
\begin{array}{lll}
\displaystyle
G_x(5,\gamma)=411\gamma^2-1909\gamma+1585>0,\;\;0<\gamma<1.
\end{array}
$$
Hence $G(x,\gamma)$ is an increasing function too, and
$ G(x,\gamma)\geq G(5,\gamma)$ if $x\geq5$.
Simple calculations yields
$$
\begin{array}{lll}
\displaystyle  G(5,\gamma)=1868\gamma^2-3068\gamma+1872\geq
G_{\min}(5,\gamma)=
G\left(5,\frac{767}{934}\right)=612\frac{131}{467},
\end{array}
$$
so, $G(x,\gamma)\geq0$.
Therefore, the following inequality holds
$$
\begin{array}{lll}
\displaystyle F(x,\gamma)=\frac{G(x,\gamma)}{486(x-2-\gamma)^2}
\geq0,
\end{array}
$$
which means $\varpi_{2,j}^{(1+\gamma)}>0$ for $\ell\geq5.$ Here we used the positivity of the coefficient
$\varpi_{1,\ell}^{(1+\gamma)}$$(\ell\geq2)$.

 Next, we show that~$\varpi_{2,\ell}^{(\alpha)}
>\varpi_{2,\ell+1}^{(\alpha)}$ for~$\ell\geq5$. Note that
$$
\begin{array}{lll}
 \displaystyle
 \varpi_{2,\ell}^{(\alpha)}-
\varpi_{2,\ell+1}^{(\alpha)}&\hspace{-0.2cm}=&\hspace{-0.3cm}\displaystyle\left(2+\gamma\right)
\left(\frac{3}{2}\right)^{1+\gamma}\left(\frac{1}{3}\right)^\ell
\sum\limits_{\ell_1=0}^{\ell}\frac{3^{\ell_1}}{\ell_1+1}
\varpi_{1,\ell_1}^{(1+\gamma)}
\varpi_{1,\ell-\ell_1}^{(1+\gamma)}\vspace{0.2 cm}\\
&&\hspace{-0.3cm}-\displaystyle\left(\frac{3}{2}\right)^
{1+\gamma}\left(\frac{1}{3}\right)^{\ell+1}
\left(1-\frac{2+\gamma}{\ell+1}\right)\varpi_{1,\ell}^{(1+\gamma)}
\vspace{0.2 cm}\\&\hspace{-0.2cm}=&\hspace{-0.3cm}\displaystyle
\left(\frac{3}{2}\right)^{1+\gamma}\left(\frac{1}{3}\right)^\ell\left[
\left(2+\gamma\right)
\sum\limits_{\ell_1=0}^{\ell}\frac{3^{\ell_1}}{\ell_1+1}
\varpi_{1,\ell_1}^{(1+\gamma)}
\varpi_{1,\ell-\ell_1}^{(1+\gamma)}-\frac{1}{3}\varpi_{1,\ell}^{(1+\gamma)}\right] \vspace{0.2 cm}\\
&&\hspace{-0.3cm}+\displaystyle \frac{2+\gamma}{\ell+1}
 \left(\frac{3}{2}\right)^
{1+\gamma}\left(\frac{1}{3}\right)^{\ell+1}
\varpi_{1,\ell}^{(1+\gamma)}
%\end{array}
%$$
%$$
%\begin{array}{lll}
\vspace{0.2 cm}\\&\hspace{-0.2cm}\geq&\hspace{-0.3cm}\displaystyle
\left(2+\gamma\right)
\left(\frac{3}{2}\right)^{1+\gamma}\left(\frac{1}{3}\right)^\ell\left[
\sum\limits_{\ell_1=0}^{\ell}\frac{3^{\ell_1}}{\ell_1+1}
\varpi_{1,\ell_1}^{(1+\gamma)}
\varpi_{1,\ell-\ell_1}^{(1+\gamma)}-\frac{1}{6}\varpi_{1,\ell}^{(1+\gamma)}\right]\vspace{0.2 cm}\\
&&\hspace{-0.3cm}+\displaystyle \frac{2+\gamma}{\ell+1}
 \left(\frac{3}{2}\right)^
{1+\gamma}\left(\frac{1}{3}\right)^{\ell+1}
\varpi_{1,\ell}^{(1+\gamma)}\vspace{0.2 cm}\\
&\hspace{-0.2cm}=&\hspace{-0.3cm}\displaystyle \left(2+\gamma\right)
\left(\frac{3}{2}\right)^{1+\gamma}\left(\frac{1}{3}\right)^\ell
P(\ell,\gamma)\varpi_{1,\ell}^{(1+\gamma)}
\displaystyle+ \frac{2+\gamma}{\ell+1}
 \left(\frac{3}{2}\right)^
{1+\gamma}\left(\frac{1}{3}\right)^{\ell+1}
\varpi_{1,\ell}^{(1+\gamma)}\vspace{0.2 cm}\\
&&\hspace{-0.3cm}+\displaystyle \left(2+\gamma\right)
\left(\frac{3}{2}\right)^{1+\gamma}\left(\frac{1}{3}\right)^\ell
\sum\limits_{\ell_1=3}^{\ell-3}\frac{3^{\ell_1}}{\ell_1+1}
\varpi_{1,\ell_1}^{(1+\gamma)} \varpi_{1,\ell-\ell_1}^{(1+\gamma)}.
\end{array}
$$
Here
$$
\begin{array}{lll}
 \displaystyle
 P(\ell,\gamma)&=&\displaystyle\left[\frac{5}{6}-\frac{3\ell(\gamma+1)}{2(\ell-2-\gamma)}
 +\frac{3\gamma(\gamma+1)\ell(\ell-1)}
 {2(\ell-2-\gamma)(\ell-3-\gamma)}\right]\vspace{0.2 cm}\\
&&+\displaystyle \frac{3^{\ell}}{\ell+1}
\left[1-\frac{(\ell+1)(\gamma+1)}{3(\ell-2-\gamma)}+\frac{\gamma(\gamma+1)\ell(\ell+1)}
 {18(\ell-2-\gamma)(\ell-3-\gamma)}\right].
\end{array}
$$

Obviously, the last two terms in the right-hand side of the last equality are both
nonnegative, so we only need prove that the factor $P(\ell,\gamma)$ in the first term
is nonnegative.

Let
$$
\begin{array}{lll}
 \displaystyle
 P_1(\ell,\gamma)=\frac{5}{6}-\frac{3\ell(\gamma+1)}{2(\ell-2-\gamma)}+\frac{3\gamma(\gamma+1)\ell(\ell-1)}
 {2(\ell-2-\gamma)(\ell-3-\gamma)},\vspace{0.3cm}\\ \displaystyle
 P_2(\ell,\gamma)=1-\frac{(\ell+1)(\gamma+1)}{3(\ell-2-\gamma)}+\frac{\gamma(\gamma+1)\ell(\ell+1)}
 {18(\ell-2-\gamma)(\ell-3-\gamma)},
\end{array}
$$
 then
$$
\begin{array}{lll}
 \displaystyle
 P(\ell,\gamma)=P_1(\ell,\gamma)+\frac{3^{\ell}}{\ell+1}
 P_2(\ell,\gamma).
\end{array}
$$

If $\ell=5$, then
$$
\begin{array}{lll}
 \displaystyle
 P_2(5,\gamma)=1-\frac{2(\gamma+1)}{(3-\gamma)}+\frac{5\gamma(\gamma+1)}
 {3(3-\gamma)(2-\gamma)}>0.
\end{array}
$$

Now we consider the case $\ell\geq6$. Let
$$
\begin{array}{lll}
 \displaystyle
 Q(x,\gamma)=18(x-2-\gamma)(x-2.5-\gamma)
 P_3(x,\gamma),\;\;x\in[6,\infty),
\end{array}
$$
where
 $$
\begin{array}{lll}
 \displaystyle
 P_3(x,\gamma)=1-\frac{(x+1)(\gamma+1)}{3(x-2-\gamma)}+\frac{\gamma(\gamma+1)x(x+1)}
 {18(x-2-\gamma)(x-2.5-\gamma)}.
\end{array}
$$
By simple calculations, one has
$$
\begin{array}{lll}
 \displaystyle
 Q_{xx}(x,\gamma)=2\gamma^2-10\gamma+24>0,\;\;0<\gamma<1.
\end{array}
$$
So $Q_{x}(x,\gamma)$~is an increasing function and
$$
\begin{array}{lll}
 \displaystyle
 Q_{x}(x,\gamma)\geq Q_{x}(6,\gamma)=19\gamma^2-80\gamma+72>0,\;x\ge6,\;0<\gamma<1.
\end{array}
$$
It immediately follows that $Q(x,\gamma)$ is an increasing function
with respect to $x$ and
$$
\begin{array}{lll}
 \displaystyle
 Q(x,\gamma)\geq Q(6,\gamma)=102\gamma^2-198\gamma+105>0,\;\;0<\gamma<1,
\end{array}
$$
i.e., $P_3(x,\gamma)>0$. Noticing
 $P_2(\ell,\gamma)>P_3(\ell,\gamma)$ yields
$$
\begin{array}{lll}
 \displaystyle
 P_2(\ell,\gamma)>0,\;\;\ell\in[5,\infty).
\end{array}
$$
Therefore,
$$
\begin{array}{lll}
 \displaystyle
 P(\ell,\gamma)&=&\displaystyle P_1(\ell,\gamma)+\frac{3^{\ell}}{\ell+1}
 P_2(\ell,\gamma)\geq P_1(\ell,\gamma)+\frac{81}{2}
 P_2(\ell,\gamma)
\vspace{0.2 cm}\\&=&\displaystyle\frac{124}{3}-\frac{3(\gamma+1)(10\ell+9)}{2(\ell-2-\gamma)}
 +\frac{3\ell\gamma(\gamma+1)(5\ell+1)}
 {4(\ell-2-\gamma)(\ell-3-\gamma)}
 \vspace{0.2 cm}\\&>&\displaystyle\frac{124}{3}-\frac{3(\gamma+1)(10\ell+9)}{2(\ell-2-\gamma)}
 +\frac{3\ell\gamma(\gamma+1)(5\ell+1)}
 {4(\ell-2-\gamma)^2}.
\end{array}
$$

When~$\ell=5$, we easily know that $P(\ell,\gamma)>0$ by direct
calculations. Next, we discuss the case $\ell\geq6$. Let
$$
\begin{array}{lll}
 \displaystyle
 P_4(\ell,\gamma)=\frac{124}{3}-\frac{3(\gamma+1)(10\ell+9)}{2(\ell-2-\gamma)}
 +\frac{3\ell\gamma(\gamma+1)(5\ell+1)}
 {4(\ell-2-\gamma)^2},
\end{array}
$$
and
$$
\begin{array}{lll}
 \displaystyle
 R(x,\gamma)=12(x-2-\gamma)^2P_4(x,\gamma),\;x\geq6.
\end{array}
$$
Differentiating twice with respect to $x$ gives
$$
\begin{array}{lll}
 \displaystyle
 R_{xx}(x,\gamma)=90\gamma^2-270\gamma+632,
\end{array}
$$
which is positive when $\gamma\in(0,1).$

So, $R_{x}(x,\gamma)$~is an increasing function and
$$
\begin{array}{lll}
 \displaystyle
 R_{x}(x,\gamma)>R_{x}(6,\gamma)=729\gamma^2-2225\gamma+2006>0.
\end{array}
$$
Furthermore, $R(x,\gamma)$ is an increasing
function as well and
$$
\begin{array}{lll}
 \displaystyle
 R(x,\gamma)>R(6,\gamma)=3412\gamma^2-6020\gamma+2968>0,\;\;x\in[6,\infty).
\end{array}
$$

So, $P_4(\ell,\gamma)>0$ implies
$P(\ell,\gamma)>0$.
It follows that $\varpi_{2,\ell}^{(\alpha)}\geq\varpi_{2,\ell+1}^{(\alpha)}$ for $\ell\geq5$.

All this completes the proof.
\end{proof}

\textbf{\emph{Appendix B}}\\

Firstly, we list a lemma which comes from \cite{dim,kjc}.\\

{\em Lemma a.}\label{le.1}
$$
\begin{array}{lll}
\displaystyle (\romannumeral1)\;\;1-x<\exp\left(-x\right)\;\;\textmd{holds}\;\;\textmd{for} \;\; 0<x<1.\vspace{0.1 cm}\\
\displaystyle (\romannumeral2)\;\;1-x>\exp\left(-2x\right)\;\;\textmd{holds}\;\; \textmd{for} \;\;0<x\leq0.7968.\vspace{0.1 cm}\\
\displaystyle (\romannumeral3)\;\;(1+x)^{\alpha}\geq 1+\alpha x\;\;\textmd{holds}\;\;\textmd{for} \;\; \alpha\leq0\;\;\textmd{or}
\;\;\alpha\geq1,\vspace{0.1 cm}\\
\displaystyle
(\romannumeral4)\;\;(1+x_1+x_2+x_3+\cdots+x_n)\leq(1+x_1)(1+x_2)(1+x_3)\cdots(1+x_s),\;\;where\vspace{0.1 cm}\\
\;\;\;\;\;\;\;\;x_m \geq-1 \;and\; \textmd{sign}(x_m)=sign(x_n)\;
for\; \forall\; 1\leq m,n\leq s.
\end{array}
$$

 Proof of Theorem \ref{th.5}.
\begin{proof}
(1). In view of (\romannumeral2) of the Lemma a, one has
$$
\begin{array}{lll}
\displaystyle
\left|\varpi_{1,\ell}^{(\alpha)}\right|&=&\displaystyle\left(1-\frac{\alpha+1}{\ell}\right)\left|\varpi_{1,\ell-1}^{(\alpha)}\right|
=\left(1-\frac{\alpha+1}{\ell}\right)\left(1-\frac{\alpha+1}{\ell-1}\right)\cdots\left(1-\frac{\alpha+1}{3}\right)\left|\varpi_{1,2}^{(\alpha)}\right|
\vspace{0.2 cm}\\
&\geq&\displaystyle\exp\left(-\frac{2(\alpha+1)}{\ell}\right)\exp\left(-\frac{2(\alpha+1)}{\ell-1}\right)\cdots
\exp\left(-\frac{2(\alpha+1)}{3}\right)\left|\varpi_{1,2}^{(\alpha)}\right|
\vspace{0.2 cm}\\
&=&\displaystyle\exp\left(-2(\alpha+1)\sum\limits_{k=3}^{\ell}\frac{1}{k}\right)\left|\varpi_{1,2}^{(\alpha)}\right|.
\end{array}
$$

Note that
$$
\begin{array}{lll}
\displaystyle
\sum\limits_{k=3}^{\ell}\frac{1}{k}<\int^\ell_2\frac{1}{x}{\rm d}x=\ln\frac{\ell}{2},
\end{array}
$$
so we get
$$
\begin{array}{lll}
\displaystyle
\left|\varpi_{1,\ell}^{(\alpha)}\right|>\exp\left(-2(\alpha+1)\ln\frac{\ell}{2}\right)\left|\varpi_{1,2}^{(\alpha)}\right|
=\frac{\alpha(1-\alpha)}{2}\left(\frac{2}{\ell}\right)^{2(\alpha+1)}.
\end{array}
$$

From \cite{dim}, we know that
$$
\begin{array}{lll}
\displaystyle
\left|\varpi_{1,\ell}^{(\alpha)}\right|<\frac{\alpha2^{\alpha+1}}{(\ell+1)^{\alpha+1}},
\end{array}
$$
i.e.,
$$
\begin{array}{lll}
\displaystyle
{B}_{1}^{L}(\alpha,\ell)
<\left|\varpi_{1,\ell}^{(\alpha)}\right|<{B}_{1}^{R}(\alpha,\ell),\;\;where\;\;{B}_{1}^{L}(\alpha,\ell)
=\frac{\alpha(1-\alpha)}{2}\left(\frac{2}{\ell}\right)^{2(\alpha+1)},\\\;{B}_{1}^{R}(\alpha,\ell)=\alpha\left(\frac{2}{\ell+1}\right)^{\alpha+1},\;\;
\ell\geq3.
\end{array}
$$

(2). From (1), we have
$$
\begin{array}{lll}
\displaystyle
\sum\limits_{k=\ell}^{\infty}\left|\varpi_{1,k}^{(\alpha)}\right|>
\alpha(1-\alpha)2^{2\alpha+1}\sum\limits_{k=\ell}^{\infty}\frac{1}{k^{2(\alpha+1)}}.
\end{array}
$$

Because of

$$
\begin{array}{lll}
\displaystyle
\sum\limits_{k=\ell}^{\infty}\frac{1}{k^{2(\alpha+1)}}>\int_\ell^\infty\frac{1}{x^{2(\alpha+1)}}{\rm d}x=\frac{1}{(2\alpha+1)\ell^{2\alpha+1}},
\end{array}
$$
then we get
$$
\begin{array}{lll}
\displaystyle
\sum\limits_{k=\ell}^{\infty}\left|\varpi_{1,k}^{(\alpha)}\right|>
\frac{\alpha(1-\alpha)}{2\alpha+1}\left(\frac{2}{\ell}\right)^{2\alpha+1}.
\end{array}
$$

In addition, we also have \cite{dim}
$$
\begin{array}{lll}
\displaystyle
\sum\limits_{k=\ell}^{\infty}\left|\varpi_{1,k}^{(\alpha)}\right|<2\left(\frac{2}{\ell}\right)^{\alpha},
\end{array}
$$
it is follows that
$$
\begin{array}{lll}
\displaystyle
 {S}_{1}^{L}(\alpha,\ell)
<\sum\limits_{k=\ell}^{\infty}\left|\varpi_{1,k}^{(\alpha)}\right|<{B}_{1}^{R}(\alpha,\ell),
\;\;where\;\; {S}_{1}^{L}(\alpha,\ell)
=\frac{\alpha(1-\alpha)}{2\alpha+1}\left(\frac{2}{\ell}\right)^{2\alpha+1},\vspace{0.2 cm}\\\displaystyle
{B}_{1}^{R}(\alpha,\ell)=2\left(\frac{2}{\ell}\right)^{\alpha},\;\;
\ell\geq3.
\end{array}
$$
\end{proof}

 Proof of Theorem \ref{th.6}.

\begin{proof}
(1) Let
$$
\begin{array}{lll}
\displaystyle
\alpha_\ell=\frac{12\ln\frac{\ell}{2}}{2\pi^2-15}-1.
\end{array}
$$
Then from $\alpha<\alpha_\ell$, we can get
$$
\begin{array}{lll}
\displaystyle
\left(\frac{\ell}{2}\right)^{1+\alpha}\exp\left(-(\alpha+1)^2\left(\frac{\pi^2}{6}-\frac{5}{4}\right)\right)>1,
\end{array}
$$
i.e.,
$$
\begin{array}{lll}
\displaystyle
\frac{{B}_{1}^{L}(\alpha,\ell)}{{\widetilde{B}}_{1}^{L}(\alpha,\ell)}<1.
\end{array}
$$
Obviously, the conditions for the above inequality are
$$
\begin{array}{lll}
\displaystyle
\alpha_3\approx0.0267,\;\alpha_4\approx0.7551,\;and\;\alpha_5\geq\frac{5}{2}\ln\frac{5}{2}-1>\frac{9}{4}-1>1.
\end{array}
$$
(2) Let
$$
\begin{array}{lll}
\displaystyle
f(\alpha,\ell)=\frac{\widetilde{B}_{1}^{L}(\alpha,\ell)}{{B}_{1}^{L}
(\alpha,\ell)}=\frac{2\alpha+1}{5\alpha}\left(\frac{\ell}{2}\right)^{1+\alpha},
\end{array}
$$
and
$$
\begin{array}{lll}
\displaystyle
g(\alpha,\ell)=\ln f(\alpha,\ell)=\ln(2\alpha+1)-\ln(5\alpha)+(1+\alpha)\ln\frac{\ell}{2}.
\end{array}
$$

Then
$$
\begin{array}{lll}
\displaystyle
g_{\alpha}(\alpha,3)=\frac{2}{2\alpha+1}-\frac{1}{\alpha}+\ln\frac{3}{2},
\end{array}
$$
and
$$
\begin{array}{lll}
\displaystyle
g_{\alpha\alpha}(\alpha,3)=\frac{4\alpha+1}{\alpha^2(2\alpha+1)^2}>0,
\end{array}
$$
that is to say that the function $g(\alpha,3)$ has a minimum value and from the
 equation $g_{\alpha}(\alpha,3)=0$ we know that the minimum point is
$$
\begin{array}{lll}
\displaystyle
\alpha^{\star}=\frac{-1+\sqrt{1+\frac{8}{\ln\frac{3}{2}}}}{4}.
\end{array}
$$

Therefore,
$$
\begin{array}{lll}
\displaystyle
f(\alpha,\ell)>f(\alpha,3)\geq f(\alpha^{\star},3)=\exp(g(\alpha^{\star},3))
=\frac{2\alpha^{\star}+1}{5\alpha^{\star}}\left(\frac{3}{2}\right)^{1+\alpha^{\star}}\approx1.3443>1,
\end{array}
$$
i.e.,
${{B}_{1}^{L}
(\alpha,\ell)}<
{\widetilde{B}_{1}^{L}(\alpha,\ell)}$.
\end{proof}

 Proof of Theorem \ref{th.7} is almost the same as that of Theorem \ref{th.5}.
\vspace{0.2cm}\\

  Proof of Theorem \ref{th.8}.
\begin{proof}
(1) Note that
$\varpi_{2,\ell}^{(\alpha)}\leq0$ for $\ell\geq3$ and
$\varpi_{1,\ell}^{(\alpha)}\leq0$ for $\ell\geq1$ ($0<\alpha<1$),
then
$$
\begin{array}{lll}
\displaystyle
\left|\varpi_{2,\ell}^{(\alpha)}\right|&=&\displaystyle\left(\frac{3}{2}\right)^{\alpha}\left|\sum\limits_{\ell_1=0}^{\ell}
\left(\frac{1}{3}\right)^{\ell_1}\varpi_{1,\ell_1}^{(\alpha)}
\varpi_{1,\ell-\ell_1}^{(\alpha)}\right|\vspace{0.2 cm}\\\displaystyle
&=&\displaystyle\left(\frac{3}{2}\right)^{\alpha}\left|\left(1+\left(\frac{1}{3}\right)^{\ell}\right)
\varpi_{1,\ell}^{(\alpha)}+ \sum\limits_{\ell_1=1}^{\ell-1}
\left(\frac{1}{3}\right)^{\ell_1}\varpi_{1,\ell_1}^{(\alpha)}
\varpi_{1,\ell-\ell_1}^{(\alpha)}\right|
\vspace{0.2 cm}\\\displaystyle
&=&\displaystyle\left(\frac{3}{2}\right)^{\alpha}\left[\left(1+\left(\frac{1}{3}\right)^{\ell}\right)
\left|\varpi_{1,\ell}^{(\alpha)}\right|-
\sum\limits_{\ell_1=1}^{\ell-1}
\left(\frac{1}{3}\right)^{\ell_1}\left|\varpi_{1,\ell_1}^{(\alpha)}\right|\left|
\varpi_{1,\ell-\ell_1}^{(\alpha)}\right|\right].
\end{array}
$$

From Lemma a, we know that
$$
\begin{array}{lll}
\displaystyle
(1+\ell_1)^{1+\alpha}(1+(\ell-\ell_1))^{1+\alpha}\geq(1+(\alpha+1)\ell_1)(1+(\alpha+1)(\ell-\ell_1))
\geq1+(\alpha+1)\ell
\end{array}
$$
and
$$
\begin{array}{lll}
\displaystyle
\ell_1^{2(1+\alpha)}(\ell-\ell_1)^{2(1+\alpha)}\leq\left(\left(\frac{\ell_1+(\ell-\ell_1)}{2}\right)^2\right)^{2(1+\alpha)}
=\left(\frac{\ell}{2}\right)^{4(\alpha+1)}.
\end{array}
$$
So,
$$
\begin{array}{lll}
\displaystyle \left|\varpi_{2,\ell}^{(\alpha)}\right|
&=&\displaystyle\left(\frac{3}{2}\right)^{\alpha}\left[\left(1+\left(\frac{1}{3}\right)^{\ell}\right)
\left|\varpi_{1,\ell}^{(\alpha)}\right|-
\sum\limits_{\ell_1=1}^{\ell-1}
\left(\frac{1}{3}\right)^{\ell_1}\left|\varpi_{1,\ell_1}^{(\alpha)}\right|\left|
\varpi_{1,\ell-\ell_1}^{(\alpha)}\right|\right]\displaystyle\vspace{0.2 cm}\\
&\leq&\displaystyle\left(\frac{3}{2}\right)^{\alpha}\left[\left(1+\left(\frac{1}{3}\right)^{\ell}\right)
\frac{\alpha2^{\alpha+1}}{(\ell+1)^{\alpha+1}}-

\sum\limits_{\ell_1=1}^{\ell-1}
\left(\frac{1}{3}\right)^{\ell_1}
\frac{\alpha^2(1-\alpha)^24^{2\alpha+1}}{(\ell-\ell_1)^{2(\alpha+1)}\ell_1^{2(\alpha+1)}}\right]
\displaystyle\vspace{0.2 cm}\\
&\leq&\displaystyle\left(\frac{3}{2}\right)^{\alpha}\left[\left(1+\left(\frac{1}{3}\right)^{\ell}\right)
\frac{\alpha2^{\alpha+1}}{(\ell+1)^{\alpha+1}}-

\frac{\alpha^2(1-\alpha)^24^{2\alpha+1}}{2}\left(1-\left(\frac{1}{3}\right)^{\ell-1}\right)
\left(\frac{2}{\ell}\right)^{4(\alpha+1)}\right],
\end{array}
$$
and
$$
\begin{array}{lll}
\displaystyle \left|\varpi_{2,\ell}^{(\alpha)}\right|
&=&\displaystyle\left(\frac{3}{2}\right)^{\alpha}\left[\left(1+\left(\frac{1}{3}\right)^{\ell}\right)
\left|\varpi_{1,\ell}^{(\alpha)}\right|-
\sum\limits_{\ell_1=1}^{\ell-1}
\left(\frac{1}{3}\right)^{\ell_1}\left|\varpi_{1,\ell_1}^{(\alpha)}\right|\left|
\varpi_{1,\ell-\ell_1}^{(\alpha)}\right|\right]\displaystyle\vspace{0.2 cm}\\
&\geq&\displaystyle\left(\frac{3}{2}\right)^{\alpha}\left[\left(1+\left(\frac{1}{3}\right)^{\ell}\right)
\frac{\alpha(1-\alpha)}{2}\left(\frac{2}{\ell}\right)^{2(1+\alpha)}\right.\vspace{0.2 cm}\\&&\displaystyle\left.
-\alpha^22^{2(\alpha+1)}\sum\limits_{\ell_1=1}^{\ell-1}
\left(\frac{1}{3}\right)^{\ell_1}
\frac{1}{(1+\ell-\ell_1)^{\alpha+1}(1+\ell_1)^{\alpha+1}}\right]
\displaystyle\vspace{0.2 cm}\\

&\geq&\displaystyle\left(\frac{3}{2}\right)^{\alpha}\left[\left(1+\left(\frac{1}{3}\right)^{\ell}\right)
\frac{\alpha(1-\alpha)}{2}\left(\frac{2}{\ell}\right)^{2(1+\alpha)}-\left(1-\left(\frac{1}{3}\right)^{\ell-1}\right)
\frac{\alpha^22^{2\alpha+1}}{1+(\alpha+1)\ell}\right].
\end{array}
$$

(2) Note that
$\varpi_{2,\ell}^{(1+\alpha)}\geq0$ for $\ell\geq4$
and $\varpi_{1,\ell}^{(1+\alpha)}\geq0$ for
$\ell\geq2$ ($0<\alpha<1$), then
$$
\begin{array}{lll}
\displaystyle
\left|\varpi_{2,\ell}^{(1+\alpha)}\right|&=&\displaystyle\left(\frac{3}{2}\right)^{1+\alpha}
\left|\sum\limits_{\ell_1=0}^{\ell}
\left(\frac{1}{3}\right)^{\ell_1}\varpi_{1,\ell_1}^{(1+\alpha)}
\varpi_{1,\ell-\ell_1}^{(1+\alpha)}\right|\vspace{0.2 cm}\displaystyle\\
%&=&\displaystyle\left(\frac{3}{2}\right)^{\alpha}\left|\left(1+\left(\frac{1}{3}\right)^{\ell}\right)
%\varpi_{1,\ell}^{(\alpha)}+ \sum\limits_{\ell_1=1}^{\ell-1}
%\left(\frac{1}{3}\right)^{\ell_1}\varpi_{1,\ell_1}^{(\alpha)}
%\varpi_{1,\ell-\ell_1}^{(\alpha)}\right|
%\vspace{0.2 cm}\displaystyle\\
&=&\displaystyle\left(\frac{3}{2}\right)^{1+\alpha}\left[\left(1+\left(\frac{1}{3}\right)^{\ell}\right)
\left|\varpi_{1,\ell}^{(1+\alpha)}\right|+
\sum\limits_{\ell_1=2}^{\ell-2}
\left(\frac{1}{3}\right)^{\ell_1}\left|\varpi_{1,\ell_1}^{(1+\alpha)}\right|\left|
\varpi_{1,\ell-\ell_1}^{(1+\alpha)}\right|\right.\vspace{0.2 cm}\displaystyle\\&&\displaystyle\left.-\left(\frac{1}{3}
+\left(\frac{1}{3}\right)^{\ell-1}\right)(1+\alpha)\left|\varpi_{1,\ell-1}^{(1+\alpha)}\right|\right].
\end{array}
$$
So,
$$
\begin{array}{lll}
\displaystyle
\left|\varpi_{2,\ell}^{(1+\alpha)}\right|&\leq&\displaystyle\left(\frac{3}{2}\right)^{1+\alpha}
\left[\left(1+\left(\frac{1}{3}\right)^{\ell}\right)
\frac{\alpha(\alpha+1)3^{\alpha+2}}{2(\ell+1)^{\alpha+2}} +
\sum\limits_{\ell_1=2}^{\ell-2} \left(\frac{1}{3}\right)^{\ell_1}
\frac{\alpha^2(1+\alpha)^23^{2(2+\alpha)}}{4(\ell_1+1)^{\alpha+2}(\ell-\ell_1+1)^{2+\alpha}}
\right.\vspace{0.2 cm}\\&&\displaystyle\left.-
\frac{(1-\alpha)\alpha(1+\alpha)^2}{6}\left(\frac{3}{\ell-1}\right)^{2(2+\alpha)}

\left(\frac{1}{3}
+\left(\frac{1}{3}\right)^{\ell-1}\right)
\right]
\vspace{0.2 cm}\\
&\leq&\displaystyle\left(\frac{3}{2}\right)^{1+\alpha}
\left[\left(1+\left(\frac{1}{3}\right)^{\ell}\right)
\frac{\alpha(\alpha+1)3^{\alpha+2}}{2(\ell+1)^{\alpha+2}} +
\left(1-\left(\frac{1}{3}\right)^{\ell-3}\right)
\frac{\alpha^2(1+\alpha)^23^{2(2+\alpha)}}{24(1+(2+\alpha)\ell)}
\right.\vspace{0.2 cm}\\&&\displaystyle\left.-\frac{(1-\alpha)\alpha(1+\alpha)^2}{6}
\left(\frac{1}{3}
+\left(\frac{1}{3}\right)^{\ell-1}\right)\left(\frac{3}{\ell-1}\right)^{2(2+\alpha)}
\right],
\end{array}
$$
and
$$
\begin{array}{lll}
\displaystyle
\left|\varpi_{2,\ell}^{(1+\alpha)}\right|&\geq&
\displaystyle\left(\frac{3}{2}\right)^{1+\alpha}
\left[\left(1+\left(\frac{1}{3}\right)^{\ell}\right)
\frac{(1-\alpha)\alpha(1+\alpha)}{6}\left(\frac{3}{\ell}\right)^{2(2+\alpha)}\right.\vspace{0.2cm}\\
&&
+\displaystyle \frac{(1-\alpha)^2\alpha^2(1+\alpha)^2}{216}\left(1-\left(\frac{1}{3}\right)^{\ell-3}\right)
\left(\frac{6}{\ell}\right)^{4(2+\alpha)}\vspace{0.2 cm}\\&&\displaystyle
-\left.\frac{\alpha(1+\alpha)^2}{2}\left(\frac{1}{3} +\left(\frac{1}{3}\right)^{\ell-1}\right)
\left(\frac{3}{\ell}\right)^{2+\alpha} \right].
\end{array}
$$

The proof is thus complete.
\end{proof}

\maketitle

\end{document}